\documentclass[11pt]{amsart}

\usepackage[top=3cm, bottom=3cm, left=2.6cm, right=2.6cm]{geometry}

\usepackage{tikz}
\usetikzlibrary{arrows, positioning}

\usepackage{graphicx,psfrag}
\usepackage{amsfonts,amsthm,amsmath,amssymb,latexsym, amscd, euscript,enumerate}
\usepackage{graphicx,color}
\usepackage[all]{xy}
\usepackage{epsfig}
\usepackage{labelfig}
\usepackage{mathrsfs}

\usepackage{epstopdf}
\theoremstyle{plain}
\newtheorem{thm}{Theorem}[section]

\newtheorem{prop}[thm]{Proposition}
\newtheorem{lem}[thm]{Lemma}

\newtheorem{rem}[thm]{Remark}

\newtheorem{definition}[thm]{Definition}

\title[Quadratic differentials and the dual pants graph]{Quadratic differentials and random walks on the dual graph of a pants decomposition}

\thanks{The first author is grateful to the Institute for Advanced Study and its unique environment. He acknowledges support from the James D. Wolfensohn Fund.}
\thanks{The second author was partially supported by a PSC-CUNY grant 68118-00 56.}
	\thanks{The third author was partially supported by National Science Foundation award DMS-2521870, Simons Collaboration Grant award 00012869, and 
 PSC-CUNY grant 67366-00 55. The third author is grateful to the Institute for Advanced Study and AMIAS
Member Fund for their support.}

\author{Charles Bordenave, Xinlong Dong and Dragomir \v Sari\' c}
\date{\today}                                          

\begin{document}

\begin{abstract}
Let $X$ be an infinite Riemann surface with an upper-bounded geodesic pants decomposition. 
The vertices of the corresponding dual graph $\mathcal{G}$ are pairs of pants and edges are cuffs with conductances equal to their lengths. We prove that the geodesic flow on $X$ is ergodic if and only if the random walk on $\mathcal{G}$ is recurrent. 
 This yields explicit criteria 
 for deciding, in terms of cuff-length growth,  whether the geodesic flow is ergodic. We provide concrete and new families of Riemann surfaces with an explicit understanding of the phase transitions from recurrent to non-recurrent geodesic flows. In addition, we show that rough isometry of surfaces does not preserve the ergodicity of the geodesic flow while rough isometry of their dual graphs does.

The above equivalence result uses
a characterization of the measured geodesic laminations on $X$ that arise as straightened horizontal foliations of finite-area holomorphic quadratic differentials. The conditions on the measured laminations are translated into the conditions on the existence of a square summable flow function on $\mathcal{G}$.
 \end{abstract}

\maketitle

\section{Introduction}

Let $X=\mathbb{H}/\Gamma$ be a Riemann surface with the Fuchsian group $\Gamma$ not finitely generated. 
The Riemann surface $X$ supports a unique conformal hyperbolic metric, and $X$ can be decomposed into a union of countably many geodesic pairs of pants with possibly adding at most countably many funnels and geodesic half-planes (see \cite{AR,BS}). The hyperbolic metric on $X$ is determined by the Fenchel-Nielsen parameters: the lengths of the boundary geodesics (called {\it cuffs}) and the twists on them.

\vspace{5pt}
{\em The type problem. } A Riemann surface $X$ that does not support a Green's function is closest to compact Riemann surfaces ($X\in O_G$). In fact, $X\in O_G$ is equivalent to many other natural properties shared by compact surfaces such as: the Brownian motion is recurrent, the Poincar\' e series for $\Gamma$ is divergent, the geodesic flow on the unit tangent bundle $T^1X$ is ergodic for the Liouville measure, the Dirichlet problem has a solution at infinity, etc (see \cite{AhlforsSario,Nicholls1,tsuji,Sullivan}). The classical type problem is to determine whether an explicitly given Riemann surface $X$ satisfies $X\in O_G$ (see \cite{AhlforsSario}). 

We consider the Fenchel-Nielsen parameters for the geodesic pants decomposition of $X$, and find a necessary and sufficient condition for $X\in O_G$. This condition is expressed in terms of a random walk on the dual graph $\mathcal{G}$ to the pants decomposition. More precisely, let $\mathcal{G} = (V,E,\ell_X)$ denote the conductance graph whose vertices $V$ are pairs of pants of the pants decomposition and whose edges $e \in E$ are the cuffs with lengths $(\ell_X(e))_{e \in E}$. The graph $\mathcal G$ (see Figure \ref{fig:graph}) may have self loops (an edge between a vertex to itself) and multiple edges (two edges between the same pair of vertices).  
\begin{figure}[h]
\leavevmode \SetLabels
\endSetLabels
\begin{center}
\AffixLabels{\centerline{\epsfig{file=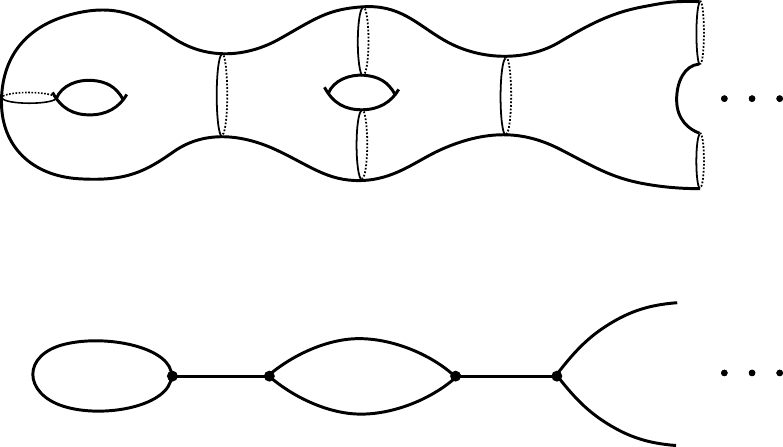,width=6.0cm,angle=0} }}
\vspace{-20pt}
\end{center}
\caption{The dual graph to the pants decomposition.} 
\label{fig:graph}
\end{figure}

For $(x,y) \in V^2$, we set $\ell_X(x,y) = \ell_X(y,x) = \sum_{e \in x \cap y} \ell_X(e) $ where the sum is over all edges $e$ between $x$ and $y$ (if $x = y$, we count $\ell_{X} (e)$ twice by convention). By construction $\ell_{X} (x,y) = 0$ unless $x$ and $y$ are adjacent in $\mathcal G$, and  $\sum_{y \in V} \ell_X(x,y) > 0$ for all $x \in V$. The random walk on $\mathcal G$  is the Markov chain on $V$ whose transition probability is given for all $x,y \in V$ by
$$
P(x,y)=\frac{\ell_X(x,y)}{\sum_{z \in V}\ell_X(x,z)}.
$$

Our main result relates $X\in O_G$ to the recurrence of the random walk on $\mathcal{G}$ weighted by the cuff lengths. In this way, recurrence of Brownian motion on $X$ and ergodicity of the geodesic flow on $T^1X$  are translated into an explicit combinatorial–geometric condition.

\begin{thm}
\label{thm:graph}
Let $X$ be an infinite Riemann surface with an upper-bounded geodesic pants decomposition and associated dual graph $\mathcal{G} = (V,E,\ell_X)$. 
The Brownian motion on $X$ is recurrent if and only if the random walk on $\mathcal{G}$ is recurrent.
\end{thm}

For surfaces with positive injectivity radius, Theorem \ref{thm:graph} is proved in \cite{Saric-ft}. We note that the conclusion of Theorem \ref{thm:graph} is false if we remove the assumption the pants decomposition has an upper-bounded geodesic lengths. The twists of cuffs can play an important role without this assumption. For example, given a surface with one topological end and an arbitrary rate of increase of boundary lengths, there are choices of twists that make the Brownian motion recurrent and other choices that make the Brownian motion transient (see \cite{HPS}).

It is interesting to compare Theorem \ref{thm:graph} with Lyons and Sullivan \cite{LyonsSullivan}. There, the authors consider a discrete locally finite cover of a connected manifold. They define a Markov chain on this cover with the property that the Brownian motion on the manifold is recurrent if and only if the Markov chain is recurrent, see with \cite[Theorem 5-6]{LyonsSullivan}. This classical result holds in much more general setting than Theorem \ref{thm:graph}. The caveat of \cite{LyonsSullivan} is that the jump transitions of the Markov chain on the cover are not explicit  and are not local. Theorem \ref{thm:graph} shows that for hyperbolic surfaces, it is possible to design such Markov chain with very explicit and local jump transitions.

Kanai \cite{Kanai} and Varopoulos \cite{Var} proved that two Riemannian manifolds with bounded Ricci curvature and positive injectivity radius that are roughly isometric simultaneously have recurrent Brownian motions or do not. We give a short proof of their theorem (see Theorem \ref{thm:kv}). 
Theorem \ref{thm:graph} allows us to understand rough isometries between Riemann surfaces without lower bounds on the lengths of closed geodesics. Somewhat unexpectedly, we establish (see Theorem \ref{thm:kanai_ce})

\begin{thm}
\label{thm:rough-isom}
Rough isometries of Riemann surfaces without a lower bound on cuff lengths do not preserve recurrence of Brownian motions.
\end{thm}

Nevertheless, we prove that rough isometries of the pair of pants decomposition preserves the type problem for Riemann surfaces  with upper-bounded geodesic pants decomposition (see \cite[Chapter 2]{LyonsPeres} for the definition of rough embeddings of graphs). 

\begin{thm}
\label{thm:rough-isom2}
Let $X,X'$ be two infinite Riemann surfaces which have an upper-bounded geodesic pants decompositions $\mathcal{G}, \mathcal{G}'$. If $\mathcal G$ and $\mathcal G'$ are roughly isometric, then the Brownian motion on $X'$ is recurrent if and only if the Brownian motion on $X$ is recurrent.
\end{thm}

Theorem \ref{thm:rough-isom2} is in sharp contrast with Theorem \ref{thm:rough-isom}. Rough isometries at the level of the pair of pants decomposition gives a finer notion of equivalence classes for Riemann surfaces than rough isometries of the surfaces.

 Theorem \ref{thm:graph} reduces the study of the type problem of Riemann surfaces with upper-bounded cuff lengths pants decomposition to the study of the type problem for the random walk on a conductance graph. This is an extremely classical question in probability theory which has been extensively studied. There are multiple tools to determine transience or recurrence of such Markov chain such as Royden's criterion, Rayleigh's monotonicity,  the Nash-Williams criterion, network reduction or rough isometries, we refer to \cite[Chapter 2]{LyonsPeres} for a modern exposition.

 \begin{figure}[h]
\leavevmode \SetLabels
\endSetLabels
\begin{center}
\AffixLabels{\centerline{\epsfig{file=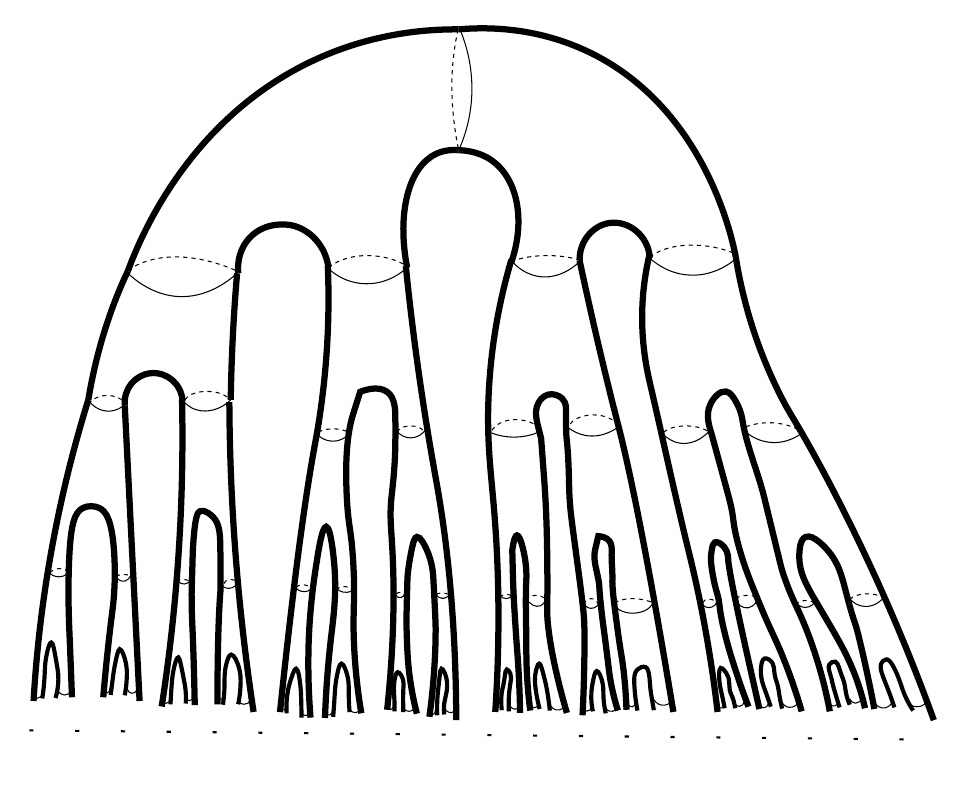,width=6.0cm,angle=0} }}
\vspace{-20pt}
\end{center}
\caption{The pants decomposition of the Cantor tree surface.} 
\label{fig:cantor}
\end{figure}

Theorem \ref{thm:graph} provides us with means to find explicit families of Riemann surfaces that are and are not in $O_G$ which complements known results. We will illustrate this on two families of surfaces in Section \ref{sec:typeex}, many more examples could be investigated. One example is a planar Riemann surface $X_C$ obtained by gluing pairs of pants such that the end space is a Cantor set, see Figure \ref{fig:cantor} for an illustration. The associated conductance graph $\mathcal G_C = (V,E,\ell_{X_C})$ is an infinite $3$-regular tree. We pick some reference cuff as an origin and, for integer $n \geq 1$, we denote by $E_n \subset E$ the subset of cuffs at graph distance $n-1$ from this reference cuff. Then the Riemann surface $X_C \in O_G$ (see \cite[Theorem 10.3]{BHS}) if for some $C>0$ and all $e \in E_n$
$$
\ell_{X_C}(e)\leq C\frac{n}{2^n}.
$$

In the opposite direction, Pandazis \cite{pan} proved that $X_C \notin O_G$ if for some $r>1$, $C_1>0$, $C_2>0$ and all $e \in E_n$
$$
C_1\frac{n^r}{2^n}\leq \ell_{X_C}(e)\leq C_2\frac{1}{n^2}.
$$
Using Theorem \ref{thm:graph}, we can give better streamlined proofs and improve significantly the above results. We will notably prove the following theorem:

 \begin{thm}
 	\label{thm:pandazis-extended}
 	Let $X_C$ be the Cantor tree surface as shown in Figure \ref{fig:cantor} such that $\sup_{e \in E} \ell_{X_C}(e)  < \infty$. If there exists a positive sequence $(\psi(n))_{n \in \mathbb N}$ with $\sum_n \psi(n)^{-1} < \infty$ such that  
  $$
  \hbox{ for all $n \geq 1$ and $e \in E_n$, }  \quad  \frac{\psi(n)}{2^n} \leq \ell_{X_C}(e) 
  $$  
    then $X \notin O_G$. Conversely, $X \in O_G$ if  $\sum_n \psi(n)^{-1} = \infty$ and 
       $$
  \hbox{ for all $n \geq 1$ and $e \in E_n$, }  \quad  \ell_{X_C}(e) \leq \frac{\psi(n)}{2^n}.
  $$    
In particular, if some $C \geq 1$, for all $n \geq 1$, $e ,  f \in E_n$, $$ \frac 1 C   \leq   \frac{\ell_{X_C}(e)}{\ell_{X_C}(f)} \leq C,$$ then $X \in O_G$ if and only if $\sum_n \ell_{X_C}(E_n) ^{-1} = \infty$, where $\ell_{X_C}(E_n)  = \sum_{e \in E_n} \ell_{X_C} (e)$.
 \end{thm}


It is interesting to study the extent to which Theorem \ref{thm:pandazis-extended} provides a classification of $O_G$. Let $(\psi(n))_n$ be a positive sequence such that $\sum_n \psi(n)^{-1} = \infty$. One question raised by Ilia Binder and Michael Yampolsky is whether some percentage of the lengths $\ell_{X_C}(e)$, $e \in E_n$, can be made larger than $\psi(n) / 2^n$ and remain in $O_G$. In this direction, we prove the following.
\begin{thm}
\label{prop:shorten}
Let $X_C$ be a Cantor tree surface as shown in Figure \ref{fig:cantor} such that $\sup_{e \in E} \ell_{X_C}(e)  <  \infty$. Let $(\psi(n))_n$ be a positive sequence such that $\sum_n \psi(n)^{-1} = \infty$. For some $\alpha>0$ and all $n \in \mathbb N$, let $K(n) = \lceil n ( \log n )^\alpha \rceil$.  Finally, for $n \geq 1$, we denote by $A_n \subset E_n$ the set of $e \in  E_n$ such that 
$$
\ell_{X_C}( e)   > \frac{\psi(n)}{2^n}.
$$
The following holds: 
\begin{enumerate}[(i)]
\item If $\alpha \leq 1 $ and $ \psi(n) + |A_n| \leq K(n)$ for all $n$ large enough, then $X \in O_G$.

\item If $\alpha >  1$, there exists $A \subset E$ with $|A \cap E_n| \leq K(n)$ such that if $\ell_{X_C}(e) = 1$ for all $e \in A$ then $X \notin O_G$.

\item There exists $A \subset E$  with $|A \cap E_n| \sim (1/3) |E_n |$ such that, if $\limsup_n \psi(n+1)/\psi(n) < \infty$ and $A_n \subset A \cap E_n$, then $X \in O_G$.
\end{enumerate}
\end{thm}

Theorem \ref{prop:shorten}  could be considered as a negative result: unexpected behavior may occur when the ratio $\max_{e,f\in E_n} \ell_{X_C} (e)/ \ell_{X_C} (f)$ diverges. Statement (ii) implies that few long cuffs can destroy parabolicity while statement (iii) implies that many long cuffs could not be sufficient to destroy recurrence. 

Note that Theorem \ref{thm:rough-isom2} implies that Theorem \ref{thm:pandazis-extended} and Theorem \ref{prop:shorten}  do extend to any infinite Riemann surface which has a pant decomposition which is roughly equivalent to the pant decomposition considered in Figure \ref{fig:cantor}.

In Section \ref{sec:typeex}, we also illustrate Theorem \ref{thm:graph} on surfaces which admit a pair of pants decomposition graph $\mathcal G$ which is roughly isometric to the standard $\mathbb Z^d$-lattice. We notably provide a new proof of a classical theorem of Lyons-Sullivan \cite{LyonsSullivan} for the type problem for the topological $\mathbb Z^d $-covers.  We also prove analogs of Theorem \ref{thm:pandazis-extended} and Theorem \ref{prop:shorten}  in this case also. We postpone the precise statements to Section \ref{sec:typeex}.



\vspace{5pt}
{\em Finite-area holomorphic quadratic differential. } The connection of $X\in O_G$ to the random walk on $\mathcal{G}$ in Theorem \ref{thm:graph} is uncovered through the finite-area holomorphic quadratic differentials on $X$. 
Holomorphic quadratic differentials are fundamental objects in Teichm\"uller theory (for example, see \cite{GL,Markovic,Sagman,Wolf}), hyperbolic geometry (see \cite{AhlforsSario,Saric23}), and surface dynamics (see \cite{DHV}). On compact Riemann surfaces, the classical theorem of Hubbard and Masur \cite{HubbardMasur} asserts that every measured foliation (or, every measured lamination) is realized uniquely as the horizontal foliation of a finite-area holomorphic quadratic differential. For infinite Riemann surfaces, however, this correspondence fails in general: not every measured lamination arises from an integrable quadratic differential (see \cite{Saric-heights}), and the obstruction is closely tied to the large-scale geometry of the surface.


A finite-area holomorphic quadratic differential on $X$ induces a measured geodesic lamination by straightening the leaves of its horizontal foliation. While this construction always yields an injective map into the space of measured laminations (see Marden-Strebel \cite{MardenStrebel} and \cite{Saric-heights}), the image $ML_{int}(X)$ consists only of those laminations that satisfy additional quantitative constraints (see \cite[Theorem 1.6]{Saric23} and \S \ref{sec:qd_ml}). Our objective is to characterize this image purely in terms of hyperbolic geometry, expressed through intersection numbers with the cuffs of a pants decomposition.

Let $X$ be an infinite Riemann surface and let $\{\alpha_n\}_n$ be an arbitrary ordering of  the cuffs (i.e., boundary geodesics) of a geodesic pants decomposition of $X$. We denote by $\{\beta_n\}_n$ the family of simple closed geodesics that are pairwise disjoint, each $\beta_k$ intersects $\alpha_k$ in a minimal number of points and does not intersect $\alpha_n$ for $n\neq k$, and the length of $\beta_k$ is minimal among these choices (see Figure \ref{fig:1}). Given a measured lamination $\mu$ on $X$, we denote by $i(\mu ,\alpha_n)$ and $i(\mu ,\beta_n)$ its intersection numbers with $\alpha_n$ and $\beta_n$ respectively, which are purely hyperbolic geometry data.
Our result gives a complete characterization of the image of the space of finite-area holomorphic quadratic differentials in terms of such data.

\begin{thm}
\label{thm:realization}
Let $X$ be an infinite Riemann surface equipped with an upper-bounded geodesic pants decomposition. A measured geodesic lamination $\mu$ on $X$ belongs to $ML_{\mathrm{int}}(X)$ if and only if
$$
\sum_{n=1}^{\infty}\Big{[}\frac{i(\mu,\alpha_n)^2}{\ell_X(\alpha_n)}+i(\mu,\beta_n)^2\ell_X(\alpha_n)\Big{]} <\infty .
$$
\end{thm}

If $\liminf_{n}\ell_X(\alpha_{n})> 0$ then this theorem is proved in \cite{Saric-ft}. Being able to extend the theorem to small cuffs is the main technical advancement and it required new tools. 
To prove the necessity of the above condition, we assume that $\mu$ is a straightening of the horizontal foliation of a finite-area holomorphic quadratic differential $\varphi$ on $X$. The first part $\sum_n \frac{i(\mu,\alpha_n)^2}{\ell_X(\alpha_n)}<\infty$ follows by the estimates of the modulus of the horizontal leaves of $\varphi$ connecting the boundaries of the standard collar around $\alpha_n$ (hence the appearance of $\ell_X(\alpha_n)$). The second part $\sum_ni(\mu,\beta_n)^2\ell_X(\alpha_n)<\infty$ is more subtle and somewhat surprising (see Proposition \ref{prop:upperbdd-beta_n} and Remark \ref{rem:int_nec}). The standard collar around $\beta_n$ has a small width and yields a weak necessary condition that is not sufficient. Instead, we divide the family of horizontal arcs that intersect $\beta_n$ into two subfamilies: the one that intersect $\beta_n$ and connect to some $\alpha_k$ for $k\neq n$ and the one that connects $\beta_n$ to itself (see Proposition \ref{prop:upperbdd-beta_n} and Figure \ref{fig:lifts}). Therefore, while the intersection $i(\mu ,\beta_n)$ (which measures the twist amount of $\mu$ around $\alpha_n$) may go to infinity, the rate is tempered by the size of the cuff $\ell_X(\alpha_n)$. 

To prove the sufficiency of the above condition, we construct a partial measured foliation on $X$ whose straightening (in homotopy) gives the measured lamination $\mu$. We then prove that the partial measured foliation has a finite Dirichlet integral. By \cite[Theorem 1.6]{Saric23}, the partial measured foliation as well as the measured lamination $\mu$ is represented by the horizontal foliation of a unique finite-area holomorphic quadratic differential. The construction of the partial measured foliation from the data of $\mu$ is proceeding by constructing the partial foliation in the collar neighborhoods of $\alpha_n$ (see Lemma \ref{lem:collar-foliation} and Figure \ref{fig:subfoliation}) and connecting the collars by rectangle neighborhoods. The construction in the collars for surfaces without a lower bound on the cuff lengths is challenging since the twisting of $\mu$ can be large for small lengths of $\alpha_n$. In addition, the asymmetry in the sizes of adjacent collars necessitates connecting intervals of different sizes with partial foliation whose Dirichlet integral is controlled. We used the idea of shrinking the collars via a quasiconformal map (see Lemma \ref{lem:trapezoid-foliation} and Figure \ref{fig:shrinking}) to create space for this. Assembling these local constructions yields a global, partial measured foliation on $X$, and the estimates show that the Dirichlet integral is finite.

In contrast with the Hubbard–Masur theorem for compact surfaces and earlier realizability results for infinite surfaces of bounded geometry (see \cite{Saric-ft}), the present work shows that the presence of arbitrarily short geodesics introduces new summability phenomena that governs both realizability of laminations and global analytic behavior.

We sketch how Theorem \ref{thm:realization} is used to prove Theorem \ref{thm:graph}. Assume that the random walk on $\mathcal{G}$ is transient. Then there exists a flow function on $\mathcal{G}$ which is in the square summable function space. We define weights on the standard train-track on $X$ (corresponding to the pants decomposition) from the values of the flow function that satisfy the summability condition from Theorem \ref{thm:realization}. Then the measured lamination (given by the weights on the train-track) can be realized by the horizontal foliation of a finite-area holomorphic quadratic differential on $X$. Since the weights were given by the flow function, it follows that a non-zero set of horizontal leaves go to infinity, and thus $X\notin O_G$. 

In the opposite direction, assume that $X\notin O_G$. Then there exists a real-valued harmonic function on $X\setminus P$ that is $0$ on the boundary $\partial P$ of $P$ and going to $1$ towards the infinity $\partial_{\infty}X$ of $X$, where $P$ is a fixed pair of pants. By taking the differential of the local holomorphic extension of the harmonic function, we obtain an Abelian differential on $X\setminus P$ whose square has a finite area. 
The horizontal foliation of the Abelian differential on $X\setminus P$ induces a flow function on $\mathcal{G}$ of finite energy which implies that the random walk is transient.

Theorem \ref{thm:realization} may be of independent interest in Teichm\" uller theory since the space of finite-area holomorphic quadratic differentials on $X$ is the pre-dual space to the tangent space to the Teichm\"uller space $T(X)$ at the basepoint (see \cite{GL}). In addition, the Teichm\"uller geodesics in $T(X)$ are given by affine maps in the natural parameters of finite-area holomorphic quadratic differentials (see \cite{GL}).

\vspace{5pt}
{\em Organization of the paper. } The rest of this work is organized as follows. 
In Section \ref{sec:modulus} we gather standard properties of the modulus of curves. In Section \ref{sec:qd_ml}, we introduce the correspondence between foliations and laminations, and prove the only if part of Theorem \ref{thm:realization}. In Section \ref{sec:suff} we complete the proof of Theorem \ref{thm:realization}. 
Theorem \ref{thm:graph} is proved in Section \ref{sec:rw-type}. Theorem \ref{thm:rough-isom} and Theorem \ref{thm:rough-isom2} are proved in Section \ref{sec:rough}. We also provide a short proof of the Kanai-Varopoulos theorem for bounded pants decomposition. In Section \ref{sec:typeex}, we study example of applications of Theorem \ref{thm:graph}, we notably prove Theorem \ref{thm:pandazis-extended} and Theorem \ref{prop:shorten}.

\section{The modulus of collars}
\label{sec:modulus}

Let $\Delta$ be a curve family on a Riemann surface $X$. A {\it conformal metric} on a Riemann surface $X=\mathbb{H}/\Gamma$ is given by $\rho (z)|dz|$ in local charts, where $\rho (z)\geq 0$ is Borel measurable. A conformal metric is {\it allowable} for $\Delta$ if
$$
\int_{\delta}\rho (z)|dz|\geq 1
$$
for all $\delta\in\Delta$. 

\begin{definition}
The modulus of a curve family $\Delta$ is given by
$$
\mathrm{mod}\Delta =\inf_{\rho} \int_X\rho (z)^2dxdy,
$$
where the infimum is over all allowable metrics $\rho$ for $\Delta$.
\end{definition}

 Given a simple closed geodesic $\alpha$ on $X$, denote by $\ell_X(\alpha )$ the hyperbolic length of $\alpha$. The {\it standard collar} around $\alpha$ is defined by
$$
\mathcal{C}(\alpha )=\Big{\{} z\in X:d_{hyp}(z,\alpha )<\sinh^{-1}\frac{1}{\sinh \ell_X(\alpha )/2}\Big{\}} ,
$$
where $d_{hyp}$ is the hyperbolic distance. 

\begin{lem}[Maskit \cite{maskit}, and also \cite{BHS}]
\label{lem:collar-modulus}
Let $\alpha$ be a simple closed geodesic on a Riemann surface $X$ and let $\Delta (\alpha )$ be the family of all curves in the standard collar $\mathcal{C}(\alpha )$ that connect its two boundary components. Given $M>0$ there exists $C>0$ such that if
$\ell_X(\alpha )\leq M$ then
$$
\mathrm{mod}[\Delta (\alpha )]\leq C\ell_X(\alpha ).
$$
\end{lem}

\section{From quadratic differentials to measured laminations}
\label{sec:qd_ml}
Let $X=\mathbb{H}/\Gamma$ be an infinite Riemann surface with the Fuchsian group $\Gamma$ of the first kind. A {\it geodesic pants decomposition} of $X$ is a partition of $X$ into geodesic pairs of pants that meet along their boundary geodesics, called {\it cuffs}. Each pair of pants is allowed to have punctures as its boundary components. 

\begin{definition}
We say that $X$ has an {\it upper-bounded} pants decomposition with cuffs $\{\alpha_n\}_n$ if $\limsup_{n} \ell_X(\alpha_n) < \infty$. 
\end{definition}

Note that this definition allows that $\liminf_{n} \ell_X(\alpha_{n})=0$. 

A {\it measured (geodesic) lamination} $\mu$ on $X$ is a geodesic lamination together with a Radon measure on arcs transverse to the geodesic lamination; a leaves-preserving homotopy between two transverse arcs preserves the measures. Denote by $ML(X)$ the space of all measured lamination on $X$. 

A {\it holomorphic quadratic differential} $\varphi$ on $X$ is an assignment of a holomorphic function in charts with the invariance $\varphi (z)dz^2$. The area form $|\varphi (z)|dxdy$ is well-defined and we say that $\varphi$ has finite area if $\int_X |\varphi |<\infty$. In addition, the {\it natural parameter} of $\varphi$ is defined by $w=\int_{z_0}^z\sqrt{\varphi (z)}dz$. If domains of two natural parameters $w_1,w_2$ overlap, then $w_1=\pm w_2+const$ in the intersection of the domains. Therefore, the horizontal (vertical) leaves are well-defined and maximal extensions of the horizontal (vertical) leaves are called the horizontal (vertical) trajectories.

Let $A(X)$ be the space of all finite area holomorphic quadratic differentials on $X$. Given $\varphi\in A(X)$, denote by $\mu_{\varphi}$ the corresponding measured geodesic lamination obtained by straightening the leaves of the horizontal foliation of $\varphi$ into homotopic hyperbolic geodesics and pushing forward the transverse measure (see \cite[\S 3.2]{Saric-heights}). The map $A(X)\to ML(X)$ given by $\varphi\mapsto\mu_{\varphi}$ is injective (see \cite{Saric-heights}), but it is not onto. 

The image $ML_{int}(X)$ of $A(X)$ is characterized in terms of the measured partial foliations with finite Dirichlet integral as follows (see \cite{Saric23}). 

\begin{definition}
Let $X=\mathbb{H}/\Gamma$ be an infinite Riemann surface such that $\Gamma$ is of the first kind.
A {\it partial measured foliation} $\mathcal{F}$ on $X$ is an assignment of a collection of sets $\{ U_i\}_i$ of $X$ (which do not have to cover the whole surface $X$) and differentiable real-valued functions 
$$
v_i:U_i\to\mathbb{R}
$$
with surjective tangent map at each point. 
The sets $U_i$ are closed Jordan domains with piecewise differentiable Jordan curves on their boundaries.
By the Implicit Function Theorem, the pre-image $v_i^{-1}(c)$ for $c\in\mathbb{R}$ is either a connected differentiable arc (possibly with endpoints) or empty, and 
\begin{equation}
\label{eq:coord-change}
v_i=\pm v_j+const
\end{equation} 
on $U_i\cap U_j$. 
\end{definition}

\begin{rem}
When $U_i\cap U_j$ has non-empty interior, the condition (\ref{eq:coord-change}) is a standard condition as in the case of the transition maps for the natural parameters of holomorphic quadratic differentials. When $U_i\cap U_j=\gamma$ is a differentiable arc, the condition (\ref{eq:coord-change}) imposes the equality of the transverse measures on $\gamma$ coming from $U_i$ and $U_j$. However, we do not require that $v_i:U_i\to\mathbb{R}$ extends to a differentiable function in an open neighborhood of $\gamma$. In fact, the functions $\{v_i\}_i$ glue to a piecewise differentiable function around $\gamma$. 
\end{rem}

For a partial measured foliation $\mathcal{F}$, a {\it horizontal arc} is a curve in $X$ that is a connected union of $v_i^{-1}(c_i)$ for some finite or infinite choice of indices $i$ and real numbers $c_i$. A {\it horizontal trajectory} of $\mathcal{F}$ is a maximal horizontal arc, including the possibility of a closed trajectory.

If $\mathcal{F}$ is a proper partial foliation of $X$, then the lift $\tilde{\mathcal{F}}$ to $\mathbb{H}$ is a proper partial measured foliation of $\mathbb{H}$. Its set of horizontal leaves is invariant under $\Gamma$.
Each non-singular horizontal trajectory of $\tilde{\mathcal{F}}$ has precisely two endpoints, and it can be homotoped to a hyperbolic geodesic of $\mathbb{H}$ modulo ideal endpoints on $\partial\mathbb{H}$. 
We denote by $G(\tilde{\mathcal{F}})$ the closure of the set of geodesics obtained by replacing the non-singular horizontal trajectories of $\tilde{\mathcal{F}}$ with the hyperbolic geodesics that share the same endpoints on $\partial\mathbb{H}$. Note that $G(\tilde{\mathcal{F}})$ is a geodesic lamination in the hyperbolic plane $\mathbb{H}$.

By repeating the arguments in \cite[\S 3.2]{Saric-heights}, the transverse measure to $\mathcal{F}$  induces a transverse measure to $\tilde{\mathcal{F}}$ which in turn induces a transverse measure on $G(\tilde{\mathcal{F}})$. Denote by $\mu_{\tilde{\mathcal{F}}}$ the induced measured lamination in $\mathbb{H}$ and note that it is invariant under the action of $\Gamma$. Therefore the measured lamination $\mu_{\tilde{\mathcal{F}}}$ induces a measured lamination $\mu_{{\mathcal{F}}}$ on $X$.

We assume that the collection  $\{ U_i\}_i$ defining the partial measured foliation $\mathcal{F}$ is locally finite in $\cup_iU_i$, i.e., every compact set in $\cup_iU_i$ intersects only finitely many sets of the collection. 

The Dirichlet integral of $v_i$ is given by $\int_{U_i}[(\frac{\partial v_i}{\partial x})^2+ (\frac{\partial v_i}{\partial y})^2]dxdy$ and by (\ref{eq:coord-change}) we have
$$
\int_{U_i\cap U_j}\Big{[}\Big{(}\frac{\partial v_i}{\partial x}\Big{)}^2+ \Big{(}\frac{\partial v_i}{\partial y}\Big{)}^2\Big{]}dxdy=\int_{U_i\cap U_j}\Big{[}\Big{(}\frac{\partial v_j}{\partial x}\Big{)}^2+ \Big{(}\frac{\partial v_j}{\partial y}\Big{)}^2\Big{]}dxdy.
$$
Using the partition of unity on $X$, the Dirichlet integral $\mathcal D(\mathcal{F})$ of $\mathcal{F}$ over $X$ is well-defined. If the integration is over a subsurface $Y$ of $X$, denote the corresponding Dirichlet integral by $\mathcal D_Y(\mathcal{F})$. 

\begin{definition}
A proper partial measured foliation $\mathcal{F}$ on $X$ is called an integrable foliation if $\mathcal D(\mathcal{F})=\mathcal D_X(\mathcal{F})<\infty$.
\end{definition}

The space of all measured laminations that are homotopic to integrable proper partial measured foliations is denoted by $ML_{int}(X)$. By \cite[Theorem 1.6]{Saric23}, the image of the space of finite-area holomorphic quadratic differentials on $X$ equals $ML_{int}(X)$.

Our goal is to obtain a characterization of $ML_{int}(X)$ in terms of the lengths of the cuffs $\{\alpha_n\}_n$, which is purely a hyperbolic geometry condition. We first give a necessary condition on
a measured lamination $\mu \in ML(X)$  to be equal to $\mu_{\varphi}$ for some $\varphi\in A(X)$ in terms of its intersection numbers with the cuffs $\alpha_n$.

\begin{prop}
\label{prop:upper-bdd}
Let $X$ be an infinite Riemann surface equipped with an upper-bounded geodesic pants decomposition $\{\alpha_n\}_n$.  For $\varphi\in A(X)$, we have
\begin{equation}
\label{eq:int-upper-bdd}
\sum_{n=1}^{\infty} \frac{[i(\mu_{\varphi} ,\alpha_n)]^2}{\ell_X(\alpha_n)}<\infty .
\end{equation}
\end{prop}

\begin{proof}
Consider the horizontal foliation $h_{\varphi}$ of $\varphi$. Let $\Sigma_n$ be the family of the subarcs of the trajectories of $h_{\varphi}$ that connect the two boundary components of the standard collar $\mathcal{C}(\alpha_n)$ around $\alpha_n$. Let $w=u+iv=\int_{z_0}^{*}\sqrt{\varphi (z)}dz$ be the natural parameter for $\varphi$. Choose a measurable subset $A_n$ of $\alpha_n$ such that each leaf of $\Sigma_n$ intersects it in exactly one point. Then we have
$$
i(\mu_{\varphi},\alpha_n)\leq \int_{A_n}|dv|
$$
since the intersection number is the infimum of $\int_{\alpha} |dv|$ over all closed curves $\alpha$ homotopic to $\alpha_n$. 

Given $z\in A_n$, let $\ell_{\varphi}(z)$ be the $\varphi$-length of the horizontal arc in $\Sigma_n$ through $z$. Then, by the Cauchy-Schwarz inequality, we have
$$
i(\mu_{\varphi},\alpha_n)^2\leq \Big{(}\int_{A_n}\frac{1}{\sqrt{\ell_{\varphi}(z)}}\sqrt{\ell_{\varphi}(z)}|dv|\Big{)}^2\leq
\Big{(}\int_{A_n} \frac{|dv|}{{\ell_{\varphi}(z)}}\Big{)}\Big{(}\int_{A_n} {{\ell_{\varphi}(z)}}{|dv|}\Big{)}.
$$
By an application of Beurling's criteria (see \cite{HS}), we have
$$
\int_{A_n} \frac{|dv|}{{\ell_{\varphi}(z)}}=mod(\Sigma_n).
$$
On the other hand, $\int_{A_n} {{\ell_{\varphi}(z)}}{|dv|}$ is the total $\varphi$-area of the union of the arcs in $\Sigma_n$ which is a subset of the collar $\mathcal{C}(\alpha_n)$.
We conclude that
$$
i(\mu_{\varphi},\alpha_n)^2\leq \Big{(}\int_{\mathcal{C}(\alpha_n)}|\varphi (z)|dxdy\Big{)}\cdot \mathrm{mod}\Sigma_n.
$$
 By the monotonicity of the modulus under the inclusion (see \cite[\S IV.3]{GM}), since $\Sigma_n\subset\Delta (\alpha_n)$ we have that $\mathrm{mod}\Sigma_n\leq\mathrm{mod}\Delta (\alpha_n)$, where $\Delta (\alpha_n)$ is the curve family connecting the boundaries of the standard collar $\mathcal{C}(\alpha_n)$. By Lemma \ref{lem:collar-modulus}, we have $\mathrm{mod}\Sigma_n\leq C\ell_X(\alpha_n)$ for some constant $C>0$.  Then
$$
\frac{i(\mu_{\varphi},\alpha_n)^2}{\ell_X(\alpha_n)}\leq C\Big{(}\int_{\mathcal{C}(\alpha_n)}|\varphi (z)|dxdy\Big{)}.
$$
By adding over all $n$ and using the fact that the standard collars around cuffs of a pants decomposition are disjoint, we obtain
$$
\sum_{n=1}^{\infty}\frac{i(\mu_{\varphi},\alpha_n)^2}{\ell_X(\alpha_n)}\leq C\sum_{n=1}^{\infty}\Big{(}\int_{\mathcal{C}(\alpha_n)}|\varphi (z)|dxdy\Big{)}\leq C\int_X |\varphi |
$$
which gives 
 (\ref{eq:int-upper-bdd}) because $\varphi$ is integrable.
\end{proof}

The measured lamination $\mu_{\varphi}$ is not determined solely by knowing the intersection numbers with all $\alpha_n$ because it may twist many times around $\alpha_n$, which is not detected by $i(\mu_{\varphi},\alpha_n)$. To address this, we consider intersections with additional curves.
For each $n$, let $\beta_n$ be a simple closed geodesic in $X$ which intersects $\alpha_n$ the minimal number of times (either one or two), does not intersect any other cuff of the pants decomposition, and has minimal length among all such geodesics. We derive a necessary condition for $i(\mu_{\varphi},\beta_n)$.

\begin{prop}
\label{prop:upperbdd-beta_n}
Let $X$ be an infinite Riemann surface equipped with an upper-bounded geodesic pants decomposition $\{\alpha_n\}_n$. For $\{\beta_n\}_n$ as above and $\varphi\in A(X)$, we have
\begin{equation}
\label{eq:int-upper-bddbeta}
\sum_{n=1}^{\infty} {\ell_X(\alpha_n)}{[i(\mu_{\varphi},\beta_n )]^2}<\infty .
\end{equation}
\end{prop}

\begin{rem}
\label{rem:int_nec}
Note that if we attempt to estimate the modulus of the curve family $\Sigma_n$ of the standard collar around $\beta_n$ we would get an upper bound of the order $\frac{e^{\frac{1}{2\ell_{X}(\alpha_n)}}}{\ell_{X}(\alpha_n)}$ or $e^{\frac{1}{2{\ell_X(\alpha_n)}}}$ if we use the non-standard collar from \cite{BHS} as $\ell_X(\alpha_n)\to 0$. We use a different method that yields a better estimate because it does not include the exponential factor. 
\end{rem}

\begin{proof}
Any $\beta_n$ as above 
intersects only one cuff $\alpha_n$ of the pants decomposition. If $\beta_n$ intersects $\alpha_n$ in two points, then
 $\beta_n$ belongs to the interior of the union $P$ of two pairs of pants of the pants decomposition. If $\beta_n$ intersects $\alpha_n$ in one point, then
 $\beta_n$ belongs to the interior of the torus with geodesic boundary $P$ that is the single pair of pants glued along $\alpha_n$.
 
 \begin{figure}
\leavevmode \SetLabels
\L(.51*.49) $\alpha_n$\\
\L(.5*.8) $\beta_n$\\
\endSetLabels
\begin{center}
\AffixLabels{\centerline{\epsfig{file=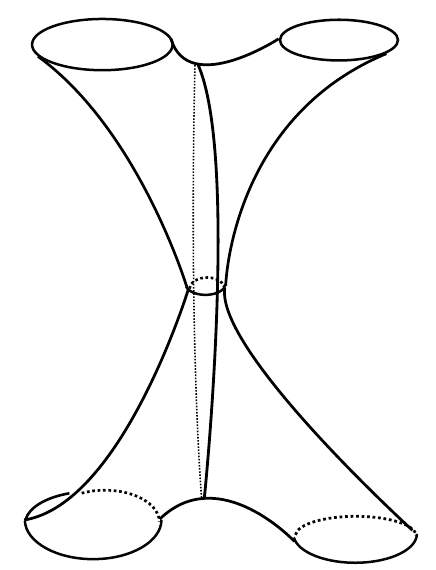,width=6.0cm,angle=0} }}
\vspace{-20pt}
\end{center}
\caption{The geodesic $\beta_n$ in $P$.} 
\label{fig:1}
\end{figure}

Assume we are in the first case. Then $\beta_n\subset P$ is diving $P$ (see Figure \ref{fig:1}). Recall that the intersection number $i(\mu_{\varphi},\beta_n)$ is the infimum over all curves homotopic to $\beta_n$ of the integral of $|Im(\sqrt{\varphi (z)}dz)|$, where a curve is allowed to have subarcs on the horizontal foliation of $\varphi$. If a subarc of a horizontal trajectory in $P$ with both endpoints on $\beta_n$ can be homotoped to $\beta_n$ modulo its endpoints, then we erase it from the set of all horizontal arcs connecting $\beta_n$ to itself since it will not contribute to the infimum of the integrals. The integral of $|Im(\sqrt{\varphi (z)}dz)|$ over the part of $\beta_n$ that intersects the remaining horizontal arcs is bounding $i(\mu_{\varphi},\beta_n)$ from the above (see \cite[\S 3.2]{Saric-heights}). We denote by $\Sigma_n$ the family of horizontal arcs of $\varphi$ in $P$ that remain after this erasing.

Let $w=u+iv$ be the natural parameter of $\varphi$ and let $B_n$ be the subset of $\beta_n$ that intersects the horizontal arcs $\Sigma_n$ which are not erased. Then, analogous to the method of the proof of Proposition \ref{prop:upper-bdd} using $B_n$ in place of $A_n$, we obtain
\begin{equation}
\label{eq:int_est}
i(\mu_{\varphi},\beta_n)^2\leq \Big{(}\int_{P}|\varphi (z)|dxdy\Big{)}\cdot \mathrm{mod}(\Sigma_n),
\end{equation}
where $\Sigma_n=\Sigma_n^1\cup\Sigma_n^2$, $\Sigma_n^1$ is the family of the horizontal arcs in $P$ that connect $\beta_n$ to itself without being homotopic to an arc on $\beta_n$ and $\Sigma_n^2$ is the family of  the horizontal arcs that connect $\beta_n$ to a boundary component of $P$.

Let $\alpha_{n_j}$ for $j=1,\ldots ,4$ be the cuffs on the boundary of $P$. Since the lengths of $\alpha_{n_j}$ are bounded from above, the collar lemma \cite[Theorem 4.1.1]{Buser} implies that the distance between $\beta_n$ and $\alpha_{n_j}$ is bounded below by a positive constant independent of $n$. It follows that
$$
\mathrm{mod}\Sigma_n^2\leq const
$$
for all $n$.

\begin{figure}[h]
	\leavevmode \SetLabels
	\endSetLabels
	\begin{center}
		\AffixLabels{\centerline{\epsfig{file =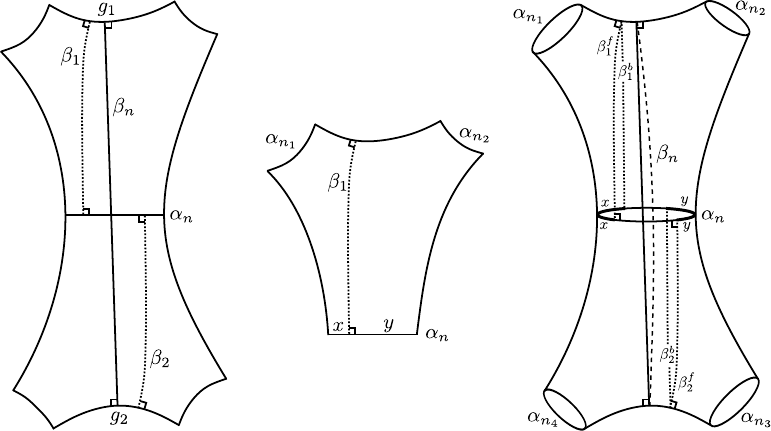, scale=0.9} }}
		\vspace{-20pt}
	\end{center}
	\caption{Front side of the union $P$ of two pairs of pants.} 
	\label{fig:front}
\end{figure}

The geodesic $\beta_n$ is contained in the union $P$ of two pairs of pants in the pants decomposition. In addition, assume that the twist along $\alpha_n$ is zero. Since the twist along $\alpha_n$ is zero, the two front-back symmetries of the two pairs of pants give a global front-back symmetry of $P$ that preserves the homotopy class of $\beta_n$. It follows that the geodesic $\beta_n$ is preserved under front-back symmetry in the union $P$ of the two pairs of pants, and therefore $\beta_n$ meets the seams $g_1$ and $g_2$ at right angles, see Figure \ref{fig:front}. Let $\beta_1$ and $\beta_2$ to be the orthogeodesics between the seams $g_1, g_2$ and $\alpha_n$, respectively. The orthogeodesics $\beta_1$ and $\beta_2$ extend by front-back symmetry to orthogeodesics in the back of $P$. The foot of the orthogeodesic $\beta_1$ (similarly for $\beta_2$) partitions the front side of $\alpha_n$ into two segments $x$ and $y$, see Figure \ref{fig:front}. We claim that there exist $M_1, M_2>0$ depending only on the upper bound of the cuff lengths $C$ such that 
\begin{equation}
\label{eq:xy}M_1\ell_X(\alpha_n) \leq x, y \leq M_2\ell_X(\alpha_n).
\end{equation}
By \cite[Theorem 7.19.2]{Bear}, we have
\begin{equation*}
	\begin{array}l
		\sinh x\sinh \frac{\ell_X (\beta_1)}{2}=\cosh\frac{\ell_X (\alpha_{n_1})}{2}\\
		\sinh y\sinh \frac{\ell_X (\beta_1)}{2}=\cosh\frac{\ell_X (\alpha_{n_2})}{2}.
	\end{array}
\end{equation*}
Dividing the two equations, we get
$$
\frac{1}{\cosh (C/2)}\leq\frac{\sinh x}{\sinh y}=\frac{\cosh\frac{\ell_X (\alpha_{n_1})}{2}}{\cosh\frac{\ell_X (\alpha_{n_2})}{2}}\leq \cosh (C/2).
$$
By $x+y=\frac{\ell_X (\alpha_n)}{2}$ and the above, the equation (\ref{eq:xy}) follows.

\begin{figure}[h]
	\leavevmode \SetLabels
	\endSetLabels
	\begin{center}
		\AffixLabels{\centerline{\epsfig{file =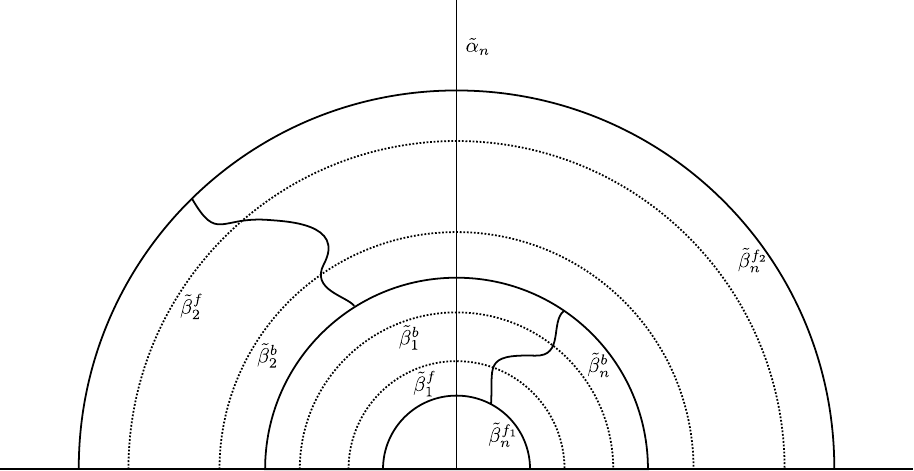, scale=0.7} }}
		\vspace{-20pt}
	\end{center}
	\caption{The lift of the family $\Sigma_n^1$ is contained in $(\Sigma_n^1)^*\cup (\Sigma_n^2)^*$.} 
	\label{fig:lifts}
\end{figure}

We claim that $\beta_n$ is located between $\beta_1$ and $\beta_2$. Indeed, the front side of $P$ in Figure \ref{fig:front} is simply connected, and it can be isometrically embedded in the universal cover of $X$. Since $\beta_n$ is orthogonal to geodesics containing $g_1$ and $g_2$, and the geodesic containing $\alpha_n$ is between them, the conclusion follows. 

Choose the universal cover of $X$ by the upper half-plane $\mathbb{H}$ such that the $y$-axis covers $\alpha_n$. Let $\tilde{\beta}_n^{f_1}, \tilde{\beta}_n^b$ and $\tilde{\beta}_n^{f_2}$ be three geodesics covering $\beta_n$ that consecutively intersect the $y$-axis. Then any arc in $\Sigma_n^1$ lifts to an arc that connects $\tilde{\beta}_n^{f_i}$ to $\tilde{\beta}_n^{b}$ for $i=1,2$, and stays between $\tilde{\beta}_n^{f_i}$ and $\tilde{\beta}_n^{b}$, see Figure \ref{fig:lifts}. 
Let $(\Sigma_n^i)^*$ be the family of all curves in between $\tilde{\beta}_n^{f_i}$ and $\tilde{\beta}_n^{b}$ that connect $\tilde{\beta}_n^{f_i}$ to $\tilde{\beta}_n^{b}$ for $i=1, 2$. By the monotonicity of the modulus (see \cite[\S IV.3]{GM}), we have
$$
\mathrm{mod}\Sigma_n^1\leq\mathrm{mod}[(\Sigma_n^1)^*]+\mathrm{mod}[(\Sigma_n^2)^*].
$$

The geodesics $\tilde{\beta}_n^{f_2}$ and $\tilde{\beta}_n^b$ are separated by two consecutive lifts of the geodesic that contains the orthogeodesic arc $\beta_2^f$ and its symmetric orthogeodesic arc $\beta_2^b$ in $P$, which we denote by $\tilde{\beta}^f_2$ and $\tilde{\beta}_2^b$ in the given order (see Figure \ref{fig:lifts}, the dotted lines). Similarly, the geodesics $\tilde{\beta}_n^{b}$ and $\tilde{\beta}_n^{f_1}$ are separated by two consecutive lifts of the geodesic that contains the orthogeodesic arc $\beta_1^b$ and its symmetric orthogeodesic arc $\beta_1^f$ in $P$, which we denote by $\tilde{\beta}^b_1$ and $\tilde{\beta}_1^f$ in the given order (see Figure \ref{fig:lifts}, the dotted lines).

Note that each curve in $(\Sigma_n^i)^*$ contains a subcurve connecting $\tilde{\beta}_i^f$ to $\tilde{\beta}_i^b$ for $i=1, 2$. Due to the front-back symmetry (see the rightmost of Figure \ref{fig:front}), the hyperbolic distance between the feet of the orthogedesic arcs $\beta_i^f$ and $\beta_i^b$ is $2x$ (or $2y$ depending on the orientation). By (\ref{eq:xy}), the hyperbolic distance between the geodesic boundaries of $(\Sigma_n^i)^*$ is at least $2M_1\ell_X(\alpha_n)$. Hence, the conformal map $z\mapsto \log z$ maps the space between the two geodesics $\tilde\beta_n^{f_i}$ to $\tilde\beta_n^b$ to the rectangle with horizontal side of length at least $2M_1\ell_X(\alpha_n)$ and vertical side of length $\pi$. The curves connecting $\tilde\beta_n^{f_i}$ to $\tilde\beta_n^b$ are mapped to the curves connecting the vertical sides of the rectangle. By the invariance under conformal maps and the monotonicity of the modulus, we obtain
$$
\mathrm{mod}\Sigma_n^1\leq\mathrm{mod}[(\Sigma_n^1)^*] +\mathrm{mod}[(\Sigma_n^2)^*]< \frac{const}{\ell_X (\alpha_n)}.
$$ 

Next, we assume that the twist along $\alpha_n$ is not zero. Since $\ell_X(\alpha_n)$ is bounded above, the map from the zero twist surface to the surface with non-zero twist is $K$-quasiconformal. Therefore, the modulus of the curve family after twisting is at most
$$K\mathrm{mod}\Sigma_n^1\leq \frac{const}{\ell_X (\alpha_n)}.
$$ 
From (\ref{eq:int_est}), we obtain
$$
\ell_X(\alpha_n) [i(\mu_{\varphi},\beta_n)]^2\leq const \int_{P}|\varphi (z)|dxdy.
$$
Since there is an upper bound on the number of geodesics $\beta_n$ that intersect any pair of pants, we conclude that $\sum_n \ell_X(\alpha_n) [i(\mu_{\varphi},\beta_n)]^2$ is bounded by a multiple of $\int_X|\varphi |$ and the sum over all $\beta_n$ in the first case is finite.

Assume that $\beta_n$ is in the interior of a torus $P$ with a geodesic boundary. Recall that $\Sigma_n^1$ is the family of the horizontal arcs of $\varphi$ that connect $\beta_n$ to itself without being homotopic to an arc on $\beta_n$, and $\Sigma_n^2$ is the family of the horizontal arcs that connect $\beta_n$ to the boundary geodesic of $P$. Note that $\mathrm{mod}\Sigma_n^2$ satisfies the same inequality as in the previous case using the same method. When the twist along $\alpha_n$ is zero, the estimate on $\mathrm{mod}\Sigma_n^1$ is much simpler. In this setting, $\beta_n$ is orthogonal to $\alpha_n$ and therefore the hyperbolic distance between two consective lifts $\tilde \beta_n^i$ for $i=1, 2$ of $\beta_n$ to the universal cover of $X$ is $\ell_X(\alpha_n)$. By the same arguments from the previous case, $\mathrm{mod}\Sigma_n^1\leq\frac{const}{\ell_X (\alpha_n)}$. Similarly, the same inequality holds when the twist along $\alpha_n$ is not zero. Thus $\sum_n \ell_X(\alpha_n) [i(\mu_{\varphi},\beta_n)]^2<\infty$ in this case as well.
\end{proof}

Note that the conditions in Propositions \ref{prop:upper-bdd} and \ref{prop:upperbdd-beta_n} can be expressed by certain conditions on the weights of a Dehn-Thurston train track that weakly carries $\mu$ (see \cite{Saric20}). This follows because for every $n$, both $i(\mu_{\varphi},\alpha_n)$ and $i(\mu_{\varphi},\beta_n)$ can be expressed as finite linear combinations of the weights on the edges in the components $P$, and conversely, each weight can be recovered from the intersection numbers. The conditions (\ref{eq:int-upper-bdd}) and (\ref{eq:int-upper-bddbeta}) are easier to use. However, we will use the train track $\Theta$ that carries $\mu_{\varphi}$ since its tangencies inform us where the leaves are going.

\section{The sufficiency of the summability condition}
\label{sec:suff}

We prove that if $\mu\in ML(X)$ satisfies the conditions in Propositions \ref{prop:upper-bdd} and \ref{prop:upperbdd-beta_n}, then $\mu$ arises from $\varphi\in A(X)$, which gives a characterization of $ML_{int}(X)$.
 
 \begin{thm}
 \label{thm:int-upper-bdd}
 Let $X$ be an infinite Riemann surface with an upper-bounded geodesic pants decomposition and let $\mu\in ML(X)$. Denote by $\{\alpha_n\}_n$ the family of cuffs and by $\{\beta_n\}_n$ the transverse family of simple  closed geodesics as above. Then
 $$
 \mu\in ML_{int}(X)
 $$ 
 if and only if 
 \begin{equation}
 \label{eq:MLint-upper}
 \sum_{n=1}^{\infty}\Big{\{} \frac{[i(\mu ,\alpha_n)]^2}{\ell_X(\alpha_n)}+ {\ell_X(\alpha_n)}{[i(\mu ,\beta_n)]^2}\Big{\}}<\infty .
 \end{equation}
 \end{thm}
 
 The ``only if'' part is established in the previous section. We need to prove the ``if'' part, which will be done by an explicit construction of a partial measured foliation that has the same intersection numbers with $\{ \alpha_n,\beta_n\}$ as $\mu$ does and whose total Dirichlet integral is comparable to the sum in (\ref{eq:MLint-upper}). To construct the partial measured foliation, we will use the geometry of the pants decomposition to decompose $X$ into pieces, constructing the foliation on each piece separately while ensuring that the transverse measures of the foliations match along the contacts between pieces. 
 
 We first construct orthogeodesic arcs in each pair of pants of the pants decomposition based on the position of the support $|\mu |$ of the measured lamination $\mu$. For any pair of pants with cuffs $\alpha_{j_1}$, $\alpha_{j_2}$, and $\alpha_{j_3}$, we draw orthogeodesic arcs between the pairs of cuffs that are intersected by the same geodesic of $|\mu |$. An orthogeodesic may go from a cuff to itself, which happens if a single geodesic of $|\mu |$ connects a single cuff to itself before leaving the pair of pants. In a single pair of pants with three cuffs, we can draw at most three orthogeodesics, and they are disjoint. If we draw fewer than three, then we add more orthogeodesics to have exactly three disjoint orthogeodesics in the given pair of pants. If a pair of pants has two or one cuff, then we can draw at most two or one orthogeodesics on the cuffs. In this process, the pants decomposition of $X$ is further divided into right-angled hexagons (when a pair of pants has three cuffs) and bigons with right angles and a single puncture.
 
  \begin{figure}[h]
 	\leavevmode \SetLabels
 	\endSetLabels
 	\begin{center}
 		\AffixLabels{\centerline{\epsfig{file =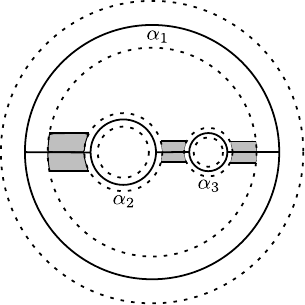, scale=1} }}
 		\vspace{-20pt}
 	\end{center}
 	\caption{Collars and connections between them.} 
 	\label{fig:connect}
 \end{figure}
 
 We draw standard collars around the cuffs and a neighborhood of the orthogeodesics of a certain width that will be determined later. Then the support $|\mu |$ of the measured geodesic lamination $\mu$ can be homotoped to this neighborhood (see Figure \ref{fig:connect}). The neighborhood can be collapsed to a train track that weakly carries $\mu$ (see \cite{Saric20}), and the weights on the edges of the train track are the transverse measure of the geodesics that are collapsed onto the edge. The edges of the train track are either arcs on the cuffs or arcs inside a pair of pants connecting cuffs. The weights on the edges connecting cuffs will inform us about the transverse measure of the partial foliation we will construct inside the neighborhood of the orthogeodesics. The foliation that we will construct inside the collar neighborhoods of the cuffs will depend on the weights on the edges of the cuff and the weights of the edges adjacent to the cuff, as these determine the amount of twisting around the cuff of the geodesic of $|\mu |$. This construction follows the outline of \cite{Saric-ft}, but it is significantly more involved since the cuff lengths approach zero, which allows the train track weights to be large and poses challenges when proving that the Dirichlet integral is finite. Therefore, we carefully construct the foliation with the above constraints.

We first consider the size of the boundary of the standard collar neighborhood of $\alpha_n$.

\begin{lem}
	\label{lem:bdry_collar_projection}
	Let $\alpha$ be a simple closed geodesic on a Riemann surface $X$. Let $\mathcal{C}(\alpha )$ be the standard collar neighborhood of $\alpha$. Let $\partial_i\mathcal{C}(\alpha )$ for $i=1,2$ be the two boundary components of $\mathcal{C}(\alpha )$. \medskip
	
	\noindent \textup{(i)} The orthogonal projection from $\partial_i\mathcal{C}(\alpha )$ onto $\alpha$ is an affine map for the distances on $\partial_i\mathcal{C}(\alpha )$ and $\alpha$ given by the hyperbolic arc-lengths. \smallskip
	
	\noindent \textup{(ii)} Given $C>0$ there exists $M>0$ such that if
	$$
	\ell_X(\alpha )\leq C
	$$
	then
	$$
	2\leq \ell_X(\partial_i\mathcal{C}(\alpha ))\leq M
	$$
	for $i=1,2$, where $\ell_X(\partial_i\mathcal{C}(\alpha ))$ is the hyperbolic length of the curve $\partial_i\mathcal{C}(\alpha )$.
\end{lem}

\begin{proof}
	Assume that the universal covering of $X$ by the upper half-plane $\mathbb H$ is such that $\alpha$ is covered by the $y$-axis and $\partial_i\mathcal{C}(\alpha)$ is covered by a ray $r$ that subtends an angle $\theta$ with the $y$-axis at $0$. Let $a, b$ be two points on $r$ and denote the hyperbolic arc length along $r$ between $a$ and $b$ by $\ell_{\mathbb H}(a, b)$. Similarly, let $a', b'$ be the two points obtained by orthogonal projections of $a, b$ respectively from $r$ onto the $y$-axis and denote the hyperbolic length along the $y$-axis between $a'$ and $b'$ by $\ell_{\mathbb H}(a', b')$. An elementary calculation shows that $\ell_{\mathbb H}(a', b') =\ell_{\mathbb H}(a, b) \frac{1}{\cos \theta} = \ell_{\mathbb H}(a, b) \cosh d$, where $d$ is the hyperbolic distance from any point on $r$ to the $y$-axis (see Beardon \S 7.20 in \cite{Bear}). Since $d$ is fixed (independent of $a,b\in r$), $\textup{(i)}$ follows.
	
	To show \textup{(ii)}, note that
	\begin{align*}
		\ell_X(\partial_i\mathcal{C}(\alpha )) 
				&= \ell_X(\alpha) \cosh \sinh^{-1} \frac{1}{\sinh\frac{1}{2} \ell_X(\alpha)}\\
		&= \ell_X(\alpha) \sqrt{1+\frac{1}{\sinh^2\frac{1}{2} \ell_X(\alpha)}}\\
		&= \ell_X(\alpha) \coth \frac{1}{2} \ell_X(\alpha).
	\end{align*} 
	Then, $\textup{(ii)}$ holds with $M={C} \coth \frac{C}{2}$ since $x \coth \frac{x}{2}$ is an increasing function and bounded below by $2$ for $x>0$.
\end{proof}

Let $\mu$ be a fixed measured lamination which satisfies (\ref{eq:MLint-upper}). We consider a homotopy that moves the support of $\mu$ into the standard collars of the cuffs $\{\alpha_n\}$ and into neighborhoods of the orthogonodesics. Then we consider the train tracks, which weakly carry $\mu$, obtained by appropriately smoothing the connection between the orthogeodesics and the cuffs in the neighborhood of the feet (see \cite{Saric20}). By smoothing at the feet, the orthogeodesic is replaced by a single edge of the train track tangent to the cuff in one direction, and all the smoothings on each side of every cuff are in the same direction (see Figure \ref{fig:smoothing}). The $\mu$-transverse measure of the support of $\mu$ that is homotoped into a neighborhood of the orthogeodesic is the {\it weight} that we assigned to the corresponding edge of the train track. Similarly, the weights on the edges on the cuffs are given by the $\mu$-transverse measure of the geodesics of the support of $\mu$ that are homotoped into the collar neighborhood of the cuff. There is a one-to-one correspondence between the measured laminations carried by a train track and the weights on the train track that satisfy switch conditions at the vertices (see \cite{Saric20}).

\begin{figure}[h]
	\leavevmode \SetLabels
	\endSetLabels
	\begin{center}
		\AffixLabels{\centerline{\epsfig{file =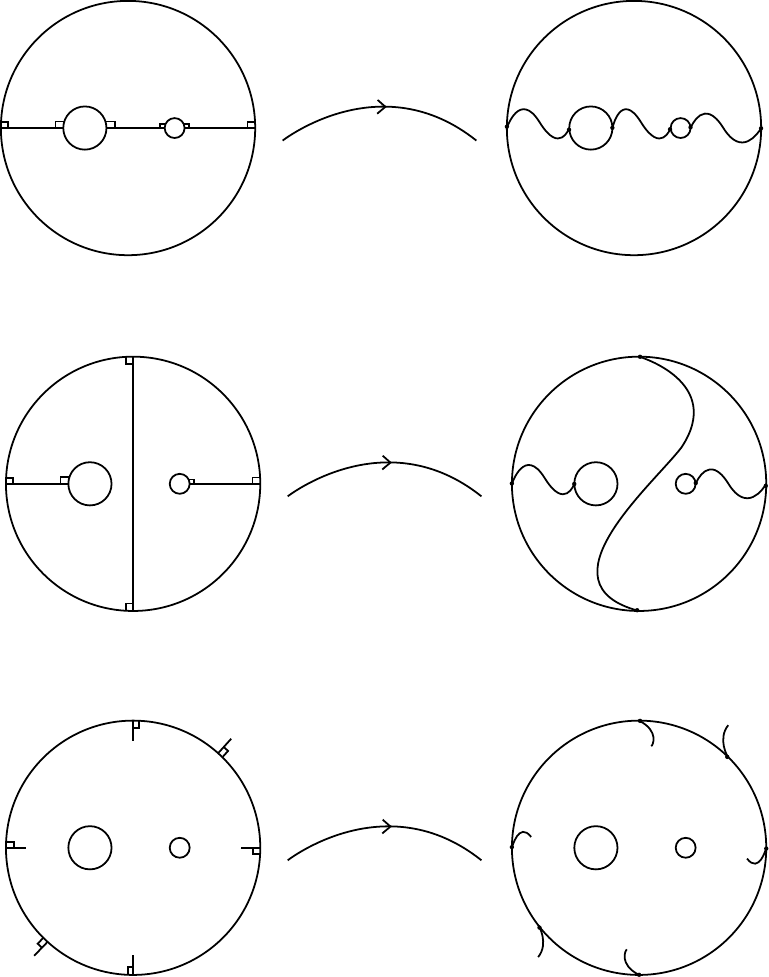, scale=0.5} }}
		\vspace{-20pt}
	\end{center}
	\caption{Smoothing of the connections at the feet of the orthogeodesics.} 
	\label{fig:smoothing}
\end{figure}

\begin{figure}[h]
	\leavevmode \SetLabels
	\endSetLabels
	\begin{center}
		\AffixLabels{\centerline{\epsfig{file =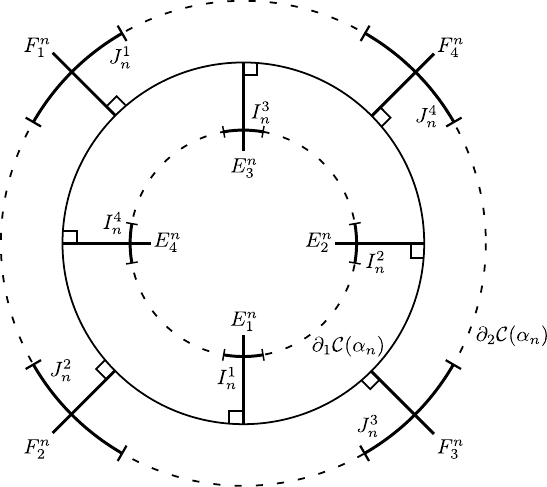, scale=0.7} }}
		\vspace{-20pt}
	\end{center}
	\caption{Orthogeodesics and collars.} 
	\label{fig:ortho}
\end{figure}

Let $e_i^n$, for $1\leq i\leq 4$, be the orthogeodesics from one side of $\alpha_n$ and $E_i^n$ be their intersection points with the boundary of the collar $\partial_1\mathcal{C}(\alpha_n)$, and let $f_j^n$, for $1\leq j\leq 4$, be the orthogeodesics from other side of $\alpha_n$ and $F_i^n$ be their intersection points with the boundary of the collar $\partial_2\mathcal{C}(\alpha_n)$ (see Figure \ref{fig:ortho}).

We define mutually disjoint subarcs $I^i_n$ of $\partial_1\mathcal{C}(\alpha_n)$ and $J^j_n$ of $\partial_2\mathcal{C}(\alpha_n)$ that are centered at the feet $E_i^n\in \partial_1\mathcal{C}(\alpha_n)$ and $F_j^n\in \partial_2\mathcal{C}(\alpha_n)$ with the property that the ratio of their lengths to $\ell_X(\partial_1\mathcal{C}(\alpha_n))$ and $\ell_X(\partial_2\mathcal{C}(\alpha_n))$ is $q>0$, respectively. 
In addition, we require that the ratio of the length of each component intervals of $\partial_1\mathcal{C}(\alpha_n))\setminus \cup_iI^i_n$ to the length of $\ell_X(\partial_1\mathcal{C}(\alpha_n))$ is at least $4q$, as well as the ratio of the length of each component intervals of $\partial_2\mathcal{C}(\alpha_n))\setminus \cup_jJ^j_n$ to $\ell_X(\partial_2\mathcal{C}(\alpha_n))$ is at least $4q$.
The constant $q>0$ is independent of $n$, and the subarcs  $I^i_n$ are pairwise disjoint, as well as the subarcs $J^j_n$. To see that such a constant $q>0$ exists, using Lemma \ref{lem:bdry_collar_projection}, it is enough to establish that the shortest distance between the feet of $e_i^n$ on $\alpha_n$ is bounded from below by $5q\ell_Y(\alpha_n)$. If the orthogeodesics are positioned as in the top line of Figure \ref{fig:smoothing}, then the two feet divide the cuff $\alpha_n$ into two equal parts by symmetry. In this case, we can take $q=\frac{1}{10}$. In the position of the second row of Figure \ref{fig:smoothing}, the inner cuffs can be of different sizes and therefore the cuff $\alpha_n$ is divided into four unequal subarcs. We establish that the ratio of their lengths is independent of $n$ in the lemma below. 

\begin{lem}
\label{lem:subarcs-size}
Let $P$ be a pair of pants with boundary components $\alpha_1$, $\alpha_2$ and $\alpha_3$ such that
$$
\ell (\alpha_i)\leq C
$$
for $i=1,2,3$, with some boundary component possibly a puncture. Consider three orthogeodesics connecting $\alpha_1$ to $\alpha_2$, 
$\alpha_1$ to $\alpha_3$, and $\alpha_1$ to itself. The cuff $\alpha_1$ is divided into four subarcs by the feet of the orthogeodesics. Then there exists $c>0$ (depending on $C$) such that the ratio of the lengths of the subarcs to the length $\ell (\alpha_1)$ is at least $c$. 
\end{lem}

\begin{proof}
Let $P$ be the pair of pants with boundary components $\alpha_i$ for $i=1,2,3$. Assume that $\alpha_1$ is a cuff while $\alpha_2$ and $\alpha_3$ are either cuffs or punctures. Draw orthogeodesics from $\alpha_1$ to $\alpha_2$, and $\alpha_1$ to $\alpha_3$. If $\alpha_2$ or $\alpha_3$ is a puncture, then the orthogeodesic from $\alpha_1$ ends in the puncture; otherwise, it is orthogonal to both cuffs. Draw the orthogeodesic from $\alpha_1$ to itself, and the orthogeodesic from $\alpha_2$ to $\alpha_3$. The pair of pants $P$ is symmetric with respect to the reflection in the three orthogeodesics connecting pairwise different boundaries, and the orthogeodesic connecting $\alpha_1$ to itself is also symmetric in this reflection. Therefore, the orthogeodesic connecting $\alpha_1$ to itself is orthogonal to the orthogeodesic connecting $\alpha_2$ to $\alpha_3$. The feet of the orthogeodesics on $\alpha_1$ divide it into four subarcs, and this division is also invariant under the reflection. Therefore, the four subarcs have lengths $x$, $y$, $y$, and $x$ in the counterclockwise order. The lemma follows from Figure \ref{fig:subarclen} and equation (\ref{eq:xy}).

\begin{figure}[h]
	\leavevmode \SetLabels
	\endSetLabels
	\begin{center}
		\AffixLabels{\centerline{\epsfig{file =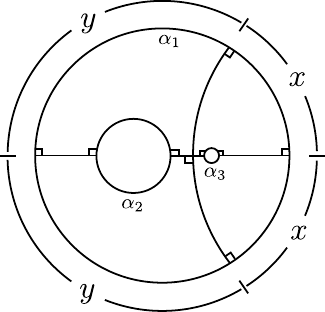, scale=0.85} }}
		\vspace{-20pt}
	\end{center}
	\caption{The lengths of subarcs on the cuffs.} 
	\label{fig:subarclen}
\end{figure}
\end{proof}

In the case of four orthogeodesics meeting a cuff from one side, we take $q=\frac{c}{5}$ where $c$ is from the above lemma. 
The lengths of $I_n^i$ are $q$ times the lengths of $\partial_1\mathcal{C}(\alpha_n)$, as well as the lengths of $J_n^j$ to be $q$ times the length of $\partial_2\mathcal{C}(\alpha_n)$. The intervals $I_n^i$ are mutually disjoint by Lemma \ref{lem:subarcs-size}, as well as the intervals $J_n^j$. Moreover, the complementary intervals to $\cup_iI_n^i$ and to $\cup_jJ_n^j$ have lengths bounded below by at least $4q\ell_X(\partial_1\mathcal{C}(\alpha_n))=4q\ell_X(\partial_2\mathcal{C}(\alpha_n))$. By Lemma \ref{lem:bdry_collar_projection}, the ratios of the lengths of the orthogonal projections of $I_n^i$ and $J_n^j$ to $\alpha_n$ over the lengths of $\alpha_n$ are also $q$. 

We denote the projections of $I_n^i$ and $J_n^j$ onto $\alpha_n$ by $P_{\alpha_n}(I_n^i)$ and $P_{\alpha_n}(J_n^j)$, respectively. In addition to the above choices, we twist around each $\alpha_n$ such that one interval of $\alpha_n\setminus \cup_i P_{\alpha_n}(I_n^i)$ and one interval of $\alpha_n\setminus \cup_j P_{\alpha_n}(J_n^j)$ match at their midpoint. 
The multi-twist map along $\{\alpha_n\}_n$ maps $X$ to a Riemann surface $Y$ by the twisting at most $\frac{1}{2}\ell_X(\alpha_n)$. Since $\ell_X(\alpha_n)$ is bounded above, the multi-twist map is homotopic to a quasiconformal map $f:X\to Y$ (see \cite{ALPS}).

If we prove that the surface $Y$ has a partial measured foliation $\mathcal{F}$ homotopic to $f_*(\mu )$ with $\mathcal{D}_Y(\mathcal{F})<\infty$, then $\mathcal{D}_X[(f^{-1})_*(\mathcal{F})]\leq K(f^{-1})\mathcal{D}_Y(\mathcal{F})<\infty$, where $K(f^{-1})$ is the quasiconformal constant of $f^{-1}$. Then, by \cite[Theorem 1.6]{Saric23}, there exists a unique finite-area holomorphic quadratic differential $\varphi\in A(X)$ such that $\mu_{\varphi}=\mu$. Therefore, it is enough to establish the statement of the theorem for $Y$.

\vskip .2 cm

The following lemma constructs a measured foliation in the collar neighborhood of $\alpha_n$ which corresponds to the part of the support of $\mu$ that intersects $\alpha_n$ and possibly twists many times around $\alpha_n$. This is the most demanding part of the construction, as the lengths of $\alpha_n$ can be arbitrarily small, which makes the collar long and allows for high twisting around $\alpha_n$ (which is reflected in the condition $\sum_ni(\mu ,\beta_n)^2\ell_Y(\alpha_n)<\infty$). 

\begin{lem}
	\label{lem:collar-foliation}
	Let $\alpha_n$ be an arbitrary cuff of the pants decomposition of $Y$ and let $\mathcal{C}(\alpha_n)$ be the standard collar with the above constructed intervals $I_n^i\subset \partial_1\mathcal{C}(\alpha_n)$ and $J_n^j\subset\partial_2\mathcal{C}(\alpha_n)$. We assume that the support of $\mu$ does not contain the cuff $\alpha_n$. Then there exists a measured foliation $\mathcal{F}_n$ of $\mathcal{C}(\alpha_n)$ whose leaves have one endpoint  $\cup_iI_n^i$ and another endpoint in $\cup_jJ_n^j$, and that realizes in the homotopy the weights of the train track corresponding to $\mu$ restricted to the edges on $\alpha_n$ and the edges with vertices in common with $\alpha_n$. Moreover, the transverse measure to each $I_n^i$ and $J_n^j$ is proportional to their hyperbolic arc length, with total mass equal to the weight of the corresponding train track edge and
	$$
	\mathcal{D}_{\mathcal{C}(\alpha_n)}(\mathcal{F}_n)\leq L\Big{[}\frac{i(\mu,\alpha_n)^2}{\ell_Y(\alpha_n)}+\ell_Y(\alpha_n)i(\mu ,\beta_n)^2\Big{]}
	$$
	for all $n$ and fixed $L>0$.
\end{lem}

\begin{proof}
	We will construct a partial foliation on $\mathcal{C}(\alpha_n)$. From the choice of the twisting above, one interval of $\alpha_n\setminus \cup_i P_{\alpha_n}(I_n^i)$ and one interval of $\alpha_n\setminus \cup_j P_{\alpha_n}(J_n^j)$ match at their midpoint, denoted by $m$.
	Then the distance from $m$ to any $P_{\alpha_n}(I_n^i)$ and any $P_{\alpha_n}(J_n^j)$ is at least $2q\ell_Y(\alpha_n)$ by the choice of $q$. 
	
	We assume that each homotoped leaf corresponding to a geodesic of $\vert \mu \vert$ intersects each geodesic arc in $\mathcal{C}(\alpha_n)$ orthogonal to $\alpha_n$ at least four times. If this not the case, we apply the appropriate fourth power of the Dehn twist around $\alpha_n$ to the measured lamination $\mu$ and obtain a new measured lamination $\mu_1$. Therefore, the total $\mu_1$-transverse measure on any geodesic arc orthogonal to $\alpha_n$ that connects $\partial_1\mathcal{C}(\alpha_n)$ and $\partial_2\mathcal{C}(\alpha_n)$ is at least three times the $\mu_1$-transverse measure of the leaves entering $\mathcal{C}(\alpha_n)$ from one side. Since the composition of these twists is a quasiconformal map, if $\mu_1$ can be realized by the horizontal foliation of a finite-area holomorphic quadratic differential, so can the original measured lamination $\mu$. By abuse of notation, we assume that $\mu$ has the above intersection property.

	Without loss of generality, we assume that the positive $ y$-axis covers $\alpha_n$ under the universal covering $\pi :\mathbb{H}\to Y$. Let $\widetilde{\mathcal{C}}(\alpha_n)$ be the set between the semicircles $\vert z \vert =1, \vert z \vert = e^{\ell_Y(\alpha_n)}$ and two Euclidean rays starting at $0$ that subtend an angle $0<\theta <\frac{\pi}{2}$ with the positive $y$-axis at $0$. When the angle $\theta$ is chosen appropriately, the set $\widetilde{\mathcal{C}}(\alpha_n)$ maps injectively (under the covering map $\pi$) to the standard collar $\mathcal{C}(\alpha_n)$ except for the identifications on the part of the boundary that lies on the semicircles. We further arrange that the intersection of semicircles $\vert z \vert =1, \vert z \vert = e^{\ell_Y(\alpha_n)}$ with the $y$-axis project to the point $m\in\alpha_n$. 
	
	For the restriction $\pi :\widetilde{\mathcal{C}}(\alpha_n) \to {\mathcal{C}}(\alpha_n)$, define $\partial_i \widetilde{\mathcal{C}}(\alpha_n)=\pi^{-1} ( \partial_i{\mathcal{C}}(\alpha_n))$, $\widetilde I^i_n=\pi^{-1}( I^i_n)$ and $\widetilde J^j_n=\pi^{-1}( J^j_n)$.
	 It follows that the lifts $\widetilde I^i_n \subset \partial_1 \widetilde{\mathcal{C}}(\alpha_n)$ of $I^i_n$ and 
	 $\widetilde J^j_n \subset \partial_2 \widetilde{\mathcal{C}}(\alpha_n)$ of $J^j_n$ lie on the two rays (see Figure \ref{fig:cover_of_collar}). 
	 
	 \begin{figure}[h]
		\leavevmode \SetLabels
		\endSetLabels
		\begin{center}
			\AffixLabels{\centerline{\epsfig{file =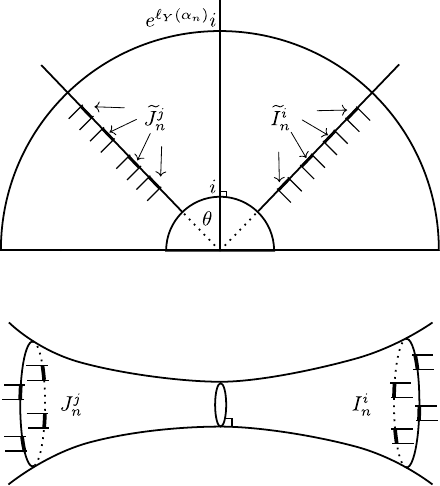, scale=1} }}
			\vspace{-20pt}
		\end{center}
		\caption{The cover of the collar.} 
		\label{fig:cover_of_collar}
	\end{figure}

	 Let $Q$ be the image of $\widetilde{\mathcal{C}}(\alpha_n)$ under $\log z$. Then $Q=[0, \ell_Y(\alpha_n)] \times [\frac{\pi}{2}-\theta, \frac{\pi}{2}+\theta]$ and the identification by the Euclidean translation of the vertical sides of $Q$ is conformal to $\mathcal{C}(\alpha_n)$. We divide $Q$ into several subdomains and define the partial foliation in each subdomain (see Figure \ref{fig:subfoliation}) such that it extends the partial foliation defined on rectangles connecting the boundaries of two collars, if any. The partial foliation of the rectangles connecting collars will be constructed later in the following lemma.
	 
	 \begin{figure}[h]
		\leavevmode \SetLabels
		\endSetLabels
		\begin{center}
			\AffixLabels{\centerline{\epsfig{file =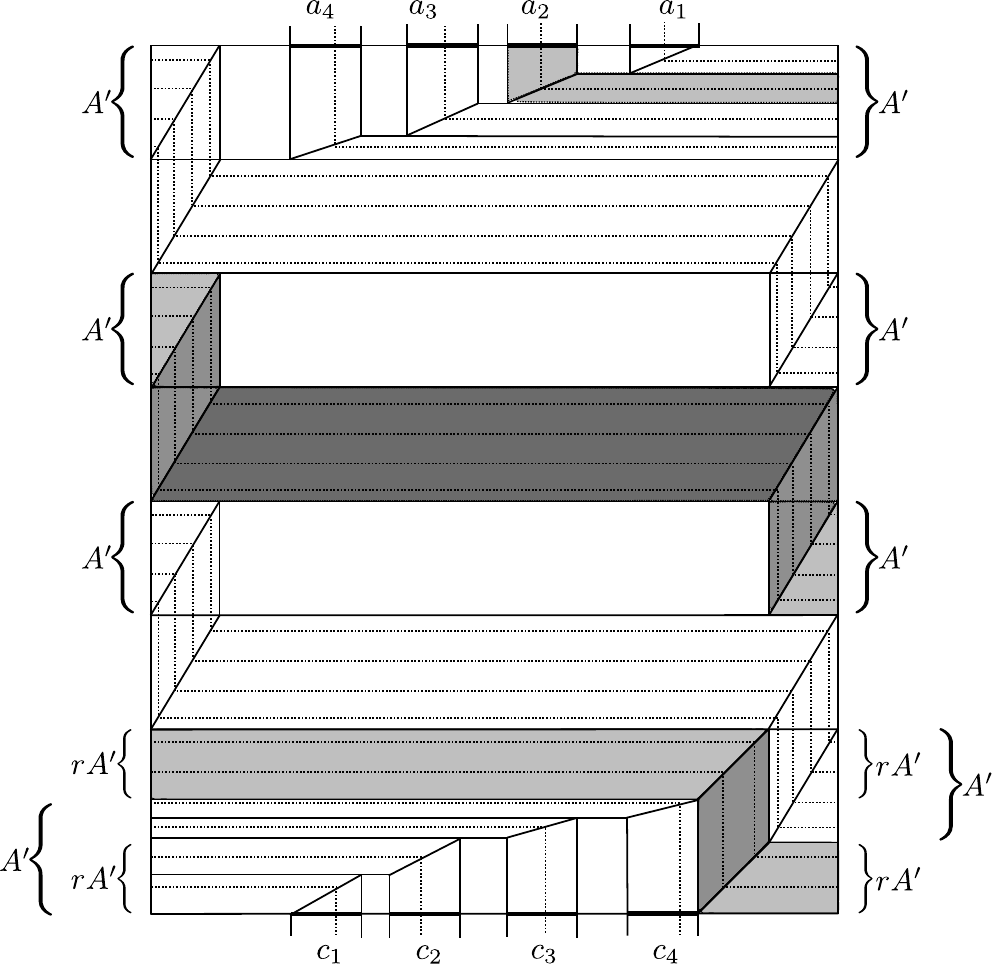, scale=0.7} }}
			\vspace{-20pt}
		\end{center}
		\caption{The partial measured foliation of $Q$.} 
		\label{fig:subfoliation}
	\end{figure}
	
	Let $b$ be the difference between the maximum of the weights of the edges of the train track on $\alpha_n$, and the sum of the weights of the edges of the train track meeting $\alpha_n$ from one side. (Note that the sum of the edges incoming to $\alpha_n$ from one side equals the sum of the weights of the edges incoming from the other sides by the switch condition.) 
	It follows that $b \leq i(\mu, \beta_n)$. We define a partial foliation of $Q$ depending on weights of the edges meeting $\alpha_n$ (i.e., the transverse measure of the support of $\mu$ that is homotoped to the neighborhoods of the orthogeodesics meeting $\mathcal{C}(\alpha_n)$). Let $\widehat I^i_n$ and $\widehat J^j_n$ be the images of $\widetilde I^i_n$ and $\widetilde J^j_n$ under the $\log z$ on the horizontal boundary sides of $Q$. The transverse measure of incoming geodesics corresponding to the interval $\widetilde I^i_n$ is denoted by $a_i$ (which is the same as the weight of the edge of the train track corresponding to the orthogedesic), and the transverse measure of incoming geodesics corresponding to the interval $\widetilde J^j_n$ is denoted by $c_j$, which is the same as the weight of the edge of the train track corresponding to the orthogedesic (see Figure \ref{fig:subfoliation}).

	In our setting, the lifts of neighborhoods of the orthogeodesics that meet $\partial_i\mathcal{C}(\alpha_n)$ under $\log z$ are rectangles that meet $Q$ at $\widehat I^i_n$ (or $\widehat J^j_n$) and their lengths are proportional to $\ell_Y(\alpha_n)$. From our previous discussion and by Lemma \ref{lem:subarcs-size}, we set the lengths to be $q\ell_Y(\alpha_n)$ where $q$ is independent of $n$. We denote the weights of the train track corresponding to the rectangles by $a_i$ (or $c_j$) and define $$A:=\sum_l a_l = \sum_s c_s,$$  $$k:=\Big{\lfloor} \frac{b}{A}\Big{\rfloor} \geq 3,$$ and $$r:=\Big{\{}\frac{b}{A}\Big{\}}.$$

	The vertical sides of $Q$ have lengths $2\theta$, and we form a partial measured foliation of $Q$ by partitioning the vertical sides into subarcs that the partial measured foliation connects and respects the gluings of the vertical sides by $z\mapsto z+\ell_Y(\alpha_n)$. We define
$$
A':=\frac{2\theta}{2k+1+r}.
$$
Then we have $2\theta =(2k+1)A'+rA'$ and we partition the vertical sides of $Q$ starting from the top into $2k+1$ subarcs of equal length $A'$ and a single subarc of length $rA'$ if $r>0$. Note that $r<1$, so the last subarc has length less than $A'$ if non-empty. 

The first interval on the right vertical side of $Q$ is used to connect all incoming foliation leaves from the intervals $\widehat I_n^i$ as follows. 
	We start our foliation from the rightmost interval $\widehat I^1_n$ by adding a right-angled triangle (part of a rectangle) denoted by $T_1$ whose top side is on $\widehat I^1_n$ with the length $q\ell_Y(\alpha_n)$ and the vertical side has length $a'_1=\frac{a_1}{\sum_la_l}A'=\frac{a_1}{A}A'$ (see Figure \ref{fig:triangle_foliation}). We define a partial foliation on $T_1$ by $v_1(x, y)=x$. The leaves are vertical lines, the transverse measure on horizontal lines is the Euclidean measure, and the Dirichlet integral satisfies $\mathcal D_{T_1}(v_1) \leq q\ell_Y(\alpha_n) \frac{a_1}{A}A'$. To make the total transverse measure equal to $a_1$, we use the function $\frac{a_1}{q\ell_Y(\alpha_n)}v_1$. The Dirichlet integral satisfies
	\begin{align}
	\label{eq:Dir(v_1)}
	\mathcal D_{T_1}(\frac{a_1}{q\ell_Y(\alpha_n)}v_1) 
	&\leq \frac{a_1^2}{q^2\ell_Y(\alpha_n)^2} q\ell_Y(\alpha_n) \frac{a_1}{A}A' \\
	&\leq C\frac{a_1^2}{\ell_Y(\alpha_n)}\leq C\frac{i(\mu ,\alpha_n)^2}{\ell_Y(\alpha_n)}, \nonumber
	\end{align}
	where $C$ is a constant independent of $n$ since $\frac{a_1}{A}\leq 1$ and $A'\leq 2\theta$. 
	The transverse measure of the foliation $\frac{a_1}{q\ell_Y(\alpha_n)}v_1$ on the hypothenuse of $T_1$ is a multiple of the Euclidean measure since the lengths on the horizontal line and the hypothenuse are proportional. 
	
	\begin{figure}[h]
		\leavevmode \SetLabels
		\endSetLabels
		\begin{center}
			\AffixLabels{\centerline{\epsfig{file=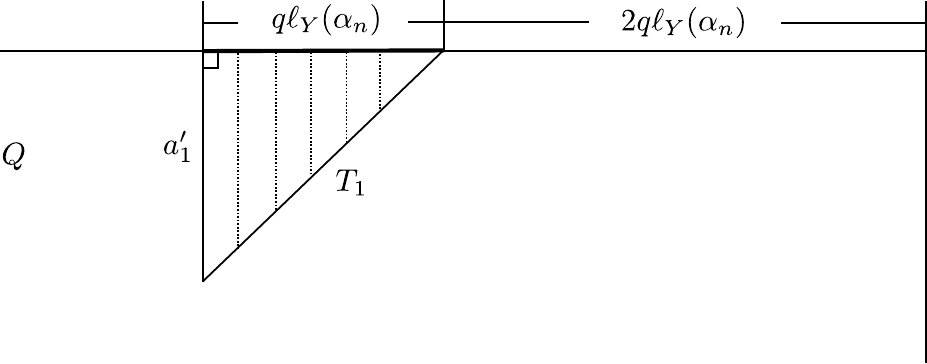, scale=0.6} }}
			\vspace{-20pt}
		\end{center}
		\caption{The partial measured foliation of a right-angled triangle $T_1$.} 
		\label{fig:triangle_foliation}
	\end{figure}

	We then extend $T_1$ by foliating a trapezoid denoted by $T'_1$ whose slanted side is in common with $T_1$, the bases are horizontal and extend to the right vertical side of $Q$, and the orthogonal side of height $a'_1$ (see Figure \ref{fig:trapezoid_fol1}). We define a partial foliation on $T'_1$ by $v'_1(x, y)=y$. The leaves are horizontal lines, the transverse measure to vertical lines is the Euclidean measure, and the Dirichlet integral satisfies 
	$$\mathcal D_{T'_1}(v'_1) \leq a_1'\ell_Y(\alpha_n) =\frac{a_1}{A}A'\ell_Y(\alpha_n).$$ 
	To achieve the total transverse measure $a_1$, we take the function $\frac{a_1}{a_1'}v_1'$. Then we have
	\begin{align}
	\label{eq:trapezoidT_1'}
	\mathcal D_{T_1'}(\frac{a_1}{a_1'}v_1')
	&\leq \frac{a_1^2}{(a_1')^2}\frac{a_1}{A}A'\ell_Y(\alpha_n)=\frac{a_1^2}{a_1^2(A')^2}A^2\frac{a_1}{A}A'\ell_Y(\alpha_n)\\ \nonumber
	&=\frac{A}{A'}a_1\ell_Y(\alpha_n) =\frac{A}{2\theta}(2k+r+1)a_1\ell_Y(\alpha_n)\\ 
	&\leq \frac{5A}{2\theta}\frac{b}{A}a_1\ell_Y(\alpha_n)\leq Cb^2\ell_Y(\alpha_n)\leq C i(\mu ,\beta_n)^2\ell_Y(\alpha_n), \nonumber
	\end{align}
where $C$ is independent of $n$ and we use the fact that $b>\sum_la_l$. The foliation defined by $\frac{a_1}{a_1'}v_1'$ has total transverse measure equal to $a_1$, and it is proportional to the Euclidean measure on the slanted side as before. Therefore, the two induced measures from the triangle and from the trapezoid are the same on the slanted side because they are both proportional to the Euclidean measure and have the same total mass. 

\begin{figure}[h]
		\leavevmode \SetLabels
		\endSetLabels
		\begin{center}
			\AffixLabels{\centerline{\epsfig{file=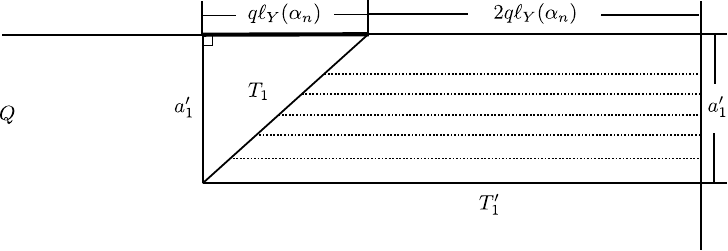, scale=0.75} }}
			\vspace{-20pt}
		\end{center}
		\caption{The partial measured foliation of a trapezoid $T_1'$.} 
		\label{fig:trapezoid_fol1}
	\end{figure}
	
	We continue this construction to the next interval $\widehat I^2_n$ by first adding a trapezoid (instead of a triangle) denoted by $T_2$ whose top side is on $\widehat I^2_n$ with the length $q\ell_Y(\alpha_n)$. Define $a_2'=\frac{a_2}{A}A'$. The lengths of the long and short vertical sides are $a'_1+a'_2$ and $a'_1$, respectively. Let $v_2(x,y)=x$ be the function that defines the foliation in the trapezoid $T_2'$ by vertical lines with transverse measure on the horizontal arcs equal to the Euclidean measure. Then we have
	$$
	\mathcal D_{T_2}(v_2)\leq q\ell_Y(\alpha_n)(a_1'+a_2')\leq  q\ell_Y(\alpha_n)A'.
	$$
To achieve the transverse measure $a_2$, we consider the function 
$\frac{a_2}{q\ell_Y(\alpha_n)}v_2$. The Dirichlet integral satisfies
\begin{equation}
\label{eq:T_2_dir}
\begin{array}l
\mathcal D_{T_2}(\frac{a_2}{q\ell_Y(\alpha_n)}v_2)\leq \frac{a_2^2}{q^2\ell_Y(\alpha_n)^2} q \ell_Y(\alpha_n)A'\leq C\frac{a_2^2}{\ell_Y(\alpha_n)}\leq 
C\frac{i(\mu ,\alpha_n)^2}{\ell_Y(\alpha_n)}.
\end{array}
\end{equation}

	\begin{figure}[h]
		\leavevmode \SetLabels
		\endSetLabels
		\begin{center}
			\AffixLabels{\centerline{\epsfig{file=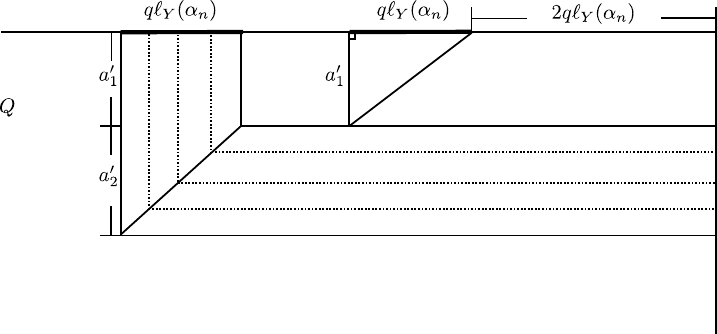, scale=0.75} }}
			\vspace{-20pt}
		\end{center}
		\caption{The partial measured foliation entering $\widehat I_n^2$.} 
		\label{fig:trapezoid_fol2}
	\end{figure}

	We then extend $T_2$ by adding a trapezoid $T'_2$ whose slanted side is in common with $T_2$, and whose height is $a'_2$. The function $\frac{a_2}{a_2'}v_2'$ with $v_2'(x,y)=y$ defines the foliation with the desired properties: the transverse measures agree on the slanted line and the total mass is $a_2$. The Dirichlet integral satisfies
	\begin{equation}
	\label{eq:T_2'_dir}
	\begin{array}l
	\mathcal D_{T_2'}(\frac{a_2}{a_2'}v_2')\leq \frac{a_2^2}{(a_2')^2}a_2'\ell_Y(\alpha_n)=\frac{a_2A}{A'}\ell_Y(\alpha_n)\leq 
	C{i(\mu ,\beta_n)^2}{\ell_Y(\alpha_n)}.
	\end{array}
	\end{equation}

	We repeat this process to all other intervals $\widehat I^i_n$ and obtain foliations whose Dirichlet integrals are bounded by the same bounds as for the first two intervals. 
	In this fashion, we cover the top subarc with length $A'$ of the right vertical side of $Q$ such that the transverse measure is proportional to the Euclidean measure and has total mass $A$.  
	The top subarc of length $A'$ on the right vertical side of $Q$ is identified by a Euclidean translation to the top subarc of length $A'$ on the left vertical side of $Q$. 

	At this point, we need to connect the top left subarc of length $A'$ to a subarc on the right vertical side of the same length $A'$ that is below the top subarc (see Figure \ref{fig:connections}). The topological constraint forces us to have the connection on the right, starting at a point at a distance $2A'$ from the top right vertex. By our definition, there are exactly $k$ connections of this type (see Figure \ref{fig:subfoliation}). Each of these connections has the same geometric shapes, and the transverse measures are all equal to $A$. We first consider the triangle $T_1$ from Figure \ref{fig:connections}. The function $\frac{A}{A'}v_1$ with $v_1(x,y)=y$ defines the foliation with horizontal leaves and the total mass $A$. We have
	$$
	\mathcal D_{T_1}(\frac{A}{A'}v_1)\leq \frac{A^2}{(A')^2}A' q \ell_Y(\alpha_n)\leq CA^2k\ell_Y(\alpha_n).
	$$
	Summing the Dirichlet integrals over all such triangles for all $k$ connections, we have that the Dirichlet integral over these pieces is less than
	$$
	CA^2k^2\ell_Y(\alpha_n)\leq Cb^2\ell_Y(\alpha_n)\leq Ci(\mu ,\beta_n)^2\ell_Y(\alpha_n).
	$$

\begin{figure}[h]
		\leavevmode \SetLabels
				\endSetLabels
		\begin{center}
			\AffixLabels{\centerline{\epsfig{file=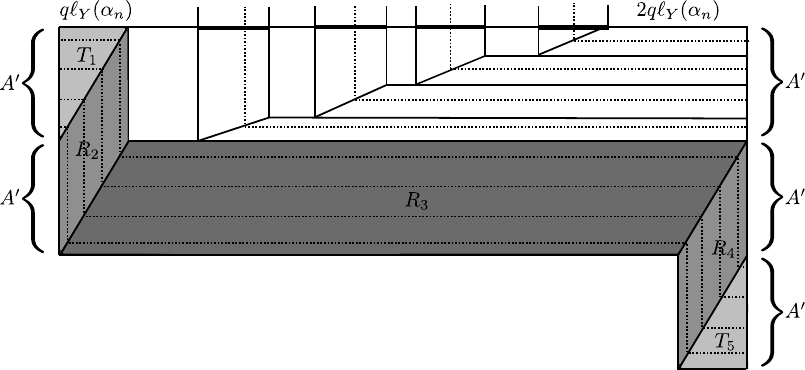, scale=0.85} }}
			\vspace{-20pt}
		\end{center}
		\caption{The partial measured foliation connecting the vertical sides of $Q$.} 
		\label{fig:connections}
	\end{figure}

The parallelogram $R_2$ from Figure \ref{fig:connections} is foliated by vertical lines using the function 
$\frac{A}{q\ell_Y(\alpha_n)}v_2$ where $v_2(x,y)=x$. We have
$$
\mathcal D_{R_2}(\frac{A}{q\ell_Y(\alpha_n)}v_2)\leq \frac{A^2}{q^2 \ell_Y(\alpha_n)^2}q\ell_Y(\alpha_n)A'\leq C \frac{A^2}{k}\frac{1}{\ell_Y(\alpha_n)}.
$$
Summing over $k$ connections, the Dirichlet integral over these pieces is less than 
$$
C\frac{A^2}{\ell_Y(\alpha_n)}\leq C\frac{i(\mu ,\alpha_n)^2}{\ell_Y(\alpha_n)}.
$$
The sum over all pieces $R_3$, $R_4$, and $T_5$, and over all $k$ connections gives the same bounds as for the above pieces.

	It remains to define the foliation that starts at the subarc of length $A'$ of the right vertical side of $Q$, a distance $2kA'$ from the top right vertex, and exits at the bottom of $Q$, according to the assigned weights $c_s$. Note that we must distribute the weight $A$ associated with the foliation over the interval of length $A'+rA'$ on the right vertical side, over the interval of the same length on the left vertical side, and exit the bottom of $Q$ in the intervals $\widehat J_n^j$ with weights $c_j$.
	
	\begin{figure}[h]
		\leavevmode \SetLabels
		\endSetLabels
		\begin{center}
			\AffixLabels{\centerline{\epsfig{file=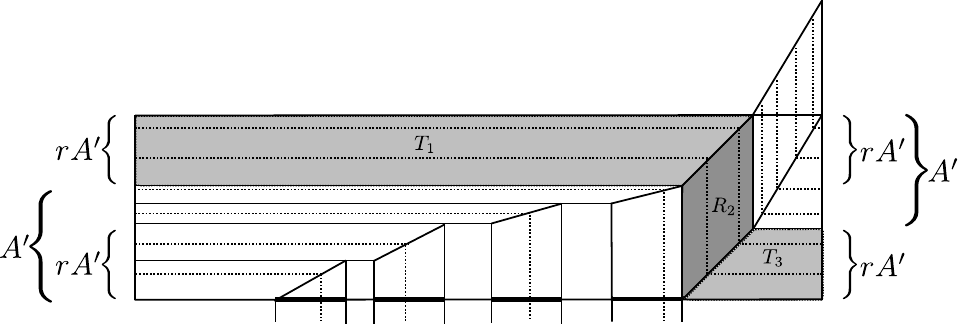, scale=0.85} }}
			\vspace{-20pt}
		\end{center}
		\caption{The partial measured foliation connecting the vertical sides of $Q$.} 
		\label{fig:bottom_fol}
	\end{figure}
	
	The foliation on the trapezoid $T_1$ (see Figure \ref{fig:bottom_fol}) is given by the function $\frac{rA}{rA'}v_1$ where $v_1(x,y)=y$. We have
	\begin{align}
	\label{eq:trap1}
	\mathcal D_{T_1}(\frac{rA}{rA'}v_1)
	&\leq \frac{A^2}{(A')^2}rA'\ell_Y(\alpha_n)\\
	&\leq \frac{A^2}{2\theta}(2k+1+r) r \ell_Y(\alpha_n)\leq C i(\mu ,\beta_n)^2\ell_Y(\alpha_n). \nonumber
	\end{align}
	
	The foliation on the parallelogram $R_2$ is given by the function 
	$\frac{rA}{q\ell_Y(\alpha_n)}v_2$ with $v_2(x,y)=x$. Then we have
	\begin{align}
	\label{eq:trap2}
	\mathcal D_{R_2}(\frac{rA}{q\ell_Y(\alpha_n)}v_2)
	&\leq \frac{r^2A^2}{q^2\ell_Y(\alpha_n)^2}2q\ell_Y(\alpha_n)A'\\
	&\leq \frac{2r^2 A^2}{q \ell_Y(\alpha_n)} A'
	\leq C \frac{i(\mu ,\alpha_n)^2}{\ell_Y(\alpha_n)}. \nonumber
	\end{align}
	The Dirichlet integral on the trapezoid $T_3$ is similar to $T_1$. Finally, the foliation connecting the bottom subarc of size $A'$ on the left vertical side of $Q$ to the intervals $\widehat J_n^j$ is similar to the top of the foliation and the estimates are the same. Collecting all the estimates, we get the statement of the lemma.
\end{proof}

Given an orthogeodesic arc connecting two cuffs $\alpha_n$ and $\alpha_m$ of the pants decomposition of $Y$, we denote by $g_{n,m}$ its subarc that connects the corresponding boundaries $\partial_1\mathcal{C}(\alpha_n )$ and $\partial_2\mathcal{C}(\alpha_m )$. Let $e_{n,m}$ be the edge of the train track corresponding to the orthogeodesic in our construction, and let $w(e_{n,m})\geq 0$ be the weight that equals the transverse measure of the part of the support of $\mu$ which homotopes to the edge. We define a partial measured foliation in the neighborhood of $g_{n,m}$, which models the weight $w(e_{n,m})$ and estimate its Dirichlet integral.

\begin{lem}
	\label{lem:connector}
	Let $g_{n,m}$ be the subarc of the orthogeodesic connecting the cuffs $\alpha_n$ and $\alpha_m$ that connects the boundaries $\partial_1\mathcal{C}(\alpha_n )$ and $\partial_2\mathcal{C}(\alpha_m )$ of their standard colar neighborhoods. We form a region $R_{n,m}$ that consists of all points at a distance of at most
	$$
	d_{n,m}=\frac{q}{2} {\min\{\ell_Y(\alpha_n),\ell_Y(\alpha_m)\}}= \frac{q}{2}{\ell_Y(\alpha_n)}
	$$
	from the orthogeodesic and between $\partial_1\mathcal{C}(\alpha_n )$ and $\partial_2\mathcal{C}(\alpha_m )$, where $q$ is the constant chosen above.
	
	Then there exists a measured foliation $\mathcal{F}_{n,m}$ on $R_{n,m}$ whose leaves are connecting $\partial_1\mathcal{C}(\alpha_n )$ and $\partial_2\mathcal{C}(\alpha_m )$, whose transverse measure is proportional to the hyperbolic arcs lengths on both  $\partial_1\mathcal{C}(\alpha_n )$ and $\partial_2\mathcal{C}(\alpha_m )$, whose total mass is $w(e_{n,m})$, and
	$$
	\mathcal{D}_R(\mathcal{F}_{n,m})\leq L\frac{i(\mu ,\alpha_n)^2}{\ell_Y(\alpha_n )}
	$$
	where $L>0$ depends only on the upper bound of the lengths of the cuffs of $Y$.
\end{lem}

\begin{rem}
The arc length measures on $\partial_1\mathcal{C}(\alpha_n )$ and $\partial_2\mathcal{C}(\alpha_m )$ are proportional to their othogonal projections onto geodesics meeting $g_{n,m}$ at a right angle.
\end{rem}

\begin{proof}
	We lift $\alpha_n, \alpha_m$ and $\partial_1\mathcal{C}(\alpha_n ), \partial_2\mathcal{C}(\alpha_m )$ to the upper half-plane $\mathbb{H}$ such that the lifts $\widetilde\alpha_n, \widetilde\alpha_m$ of $\alpha_n, \alpha_m$ respectively are orthogonal to the $y$-axis, the lifts $\widetilde\partial_1\mathcal{C}(\alpha_n ), \widetilde\partial_2\mathcal{C}(\alpha_m )$ of $\partial_1\mathcal{C}(\alpha_n ), \partial_2\mathcal{C}(\alpha_m )$ respectively are orthogonal to the $y$-axis, and the lift $\widetilde g_{n, m}$ of $g_{n, m}$ is covered by the $y$-axis (see Figure \ref{fig:orthog}). It follows that the lift of $\widetilde R_{n, m}$ of $R_{n, m}$ is between two rays that subtend an angle $\theta$ with the $y$-axis at $0$ such that (see \S 7.20 in \cite{Bear}) 
\begin{equation}
\label{eq:log_coord}
\sin \theta = \tanh d_{n, m} = \tanh {\frac{q}{2}\ell_Y(\alpha_n)}.
\end{equation}  
	
	\begin{figure}[h]
		\leavevmode \SetLabels
		\endSetLabels
		\begin{center}
			\AffixLabels{\centerline{\epsfig{file=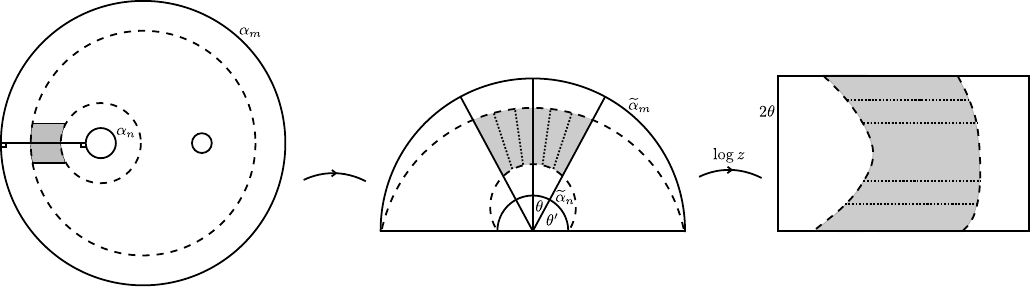, scale=0.7} }}
			\vspace{-20pt}
		\end{center}
		\caption{The foliation of a neighborhood of an ortogeodesic connecting two collars.} 
		\label{fig:orthog}
	\end{figure}

	We will construct a partial measured foliation on $\widetilde R_{n, m}$ whose leaves are equidistant lines to $\widetilde g_{n, m}$ and whose transverse measure on any geodesic arc orthogonal to $\widetilde g_{n, m}$ is the hyperbolic length on that arc. By Lemma \ref{lem:bdry_collar_projection}, the transverse measure is proportional to the hyperbolic arc lengths of both $\widetilde\partial_1\mathcal{C}(\alpha_n )$ and $\widetilde\partial_2\mathcal{C}(\alpha_m )$.
	
	We denote by $\widetilde g'_{n, m}$ the orthogeodesic arc connecting $\widetilde\alpha_n$ and $\widetilde\alpha_m$, and let $\log z = s + it$ be the logarithmic coordinates for $z \in \widetilde R_{n, m}$, where $-\frac{\ell_{\mathbb H}(\widetilde g'_{n, m})}{2} < s < \frac{\ell_{\mathbb H}(\widetilde g'_{n, m})}{2}$ and $\frac{\pi}{2}-\theta \leq t \leq \frac{\pi}{2}+\theta$. The inequalities for $s$ are strict since the image of $\widetilde R_{n, m}$ under $\log z$ does not meet the vertical sides of the rectangle $[-\frac{\ell_{\mathbb H}(\widetilde g'_{n, m})}{2}, \frac{\ell_{\mathbb H}(\widetilde g'_{n, m})}{2}] \times [\frac{\pi}{2}-\theta, \frac{\pi}{2}+\theta] $. 
	
	If $d(z)$ is the hyperbolic distance from $z=|z|e^{i\theta'}$ to the $y$-axis, then the formula (\ref{eq:log_coord}) gives 
	\begin{equation}
	\label{eq:rel}
	\vert \cos \theta' \vert = \tanh d(z). 
	\end{equation}
	Note that $\log z$ is not an isometry for the hyperbolic distances on the geodesic arcs orthogonal to $\widetilde g_{n, m}$ to the Euclidean distances in their images. Let $$f(s+it)=s-i \cdot \tanh^{-1}[\cos t]$$ be defined on $[-\frac{\ell_{\mathbb H}(\widetilde g'_{n, m})}{2}, \frac{\ell_{\mathbb H}(\widetilde g'_{n, m})}{2}] \times [\frac{\pi}{2}-\theta, \frac{\pi}{2}+\theta]$. The function $f(s,t)$ is differentiable, $f_s=1$, and $f_t=\frac{i}{\sin t}$. Therefore, the Beltrami coefficient of $f$ is
    $$
\Big{|} \frac{\frac{1}{\sin t}-1}{\frac{1}{\sin t}+1}\Big{|}
    $$
    and, 
    $f(s,t)$ is a quasiconformal map (with a quasiconformal constant independent of $n$) on its domain since $t$ is bounded away from $0$ and $\pi$ uniformly in $n$. 
	By (\ref{eq:rel}), the composition $f(\log z)$ is an isometry on the geodesic arcs orthogonal to $\widetilde g_{n, m}$. 
Therefore, the image of $\widetilde R_{n, m}$ under $f(\log z)$ denoted by $\widehat R_{n, m}$ is in a rectangle $R$ with the length $\ell_{\mathbb H}(\widetilde g'_{n, m})$ and the height $2d_{n, m}=q{\ell_Y(\alpha_n)}$.
	
	We define a partial foliation of $\widehat R_{n, m}$ by $v(s+it)=t$. The leaves are horizontal lines and the Dirichlet integral $\mathcal D_{\widehat R_{n, m}}(v) \leq C \cdot {\ell_Y(\alpha_n)}$ for a universal constant $C>0$ because $\ell_{\mathbb H}(\widetilde g_{n, m})$ is bounded from the above in terms of the upper bound of the lengths of the cuffs. 
	The pull-back of $\widetilde v_{n, m}$ on $\widetilde R_{n, m}$ defines a partial measured foliation whose leaves are as in Figure \ref{fig:orthog} and whose transverse measure is given by the hyperbolic length on the arc orthogonal to $\widetilde g_{n, m}$. The Dirichlet integral $\mathcal D_{\widetilde R_{n, m}}(\widetilde v_{n, m})$ is bounded above by a constant multiple of $\mathcal D_{\widehat R_{n, m}}(v)$ (the Dirichlet integral of the pre-composition of a function with a $K$-quasiconformal map is at most $K$ times the Dirichlet integral of the original function). By the conformal invariance of the Dirichlet integral, we denote the partial measured foliation on $R_{n, m}$ by $v_{n, m}$. We define $\mathcal F_{n, m}=\frac{w(e_{n, m})}{q\ell_Y(\alpha_n)} v_{n, m}$ such that the total transverse measure $w(e_{n, m})$ is achieved and 
	\begin{align*}
		\mathcal D_{R_{n, m}}(\mathcal F_{n, m})
		&=\bigg[\frac{w(e_{n, m})}{q\ell_Y(\alpha_n)}\bigg]^2 \mathcal D_{\widetilde R_{n, m}}(\widetilde v_{n, m}) \\
		&\leq L \cdot \frac{[w(e_{n, m})]^2}{\ell_Y(\alpha_n)} = L \cdot \frac{i(\mu ,\alpha_n)^2}{\ell_Y(\alpha_n )}
	\end{align*} where $L>0$ depends only on the upper bound of the lengths of the cuffs of $Y$.
\end{proof}

The following lemma constructs a partial measured foliation of a Euclidean trapezoid whose one side is orthogonal to both base sides. The leaves of the foliation connect the bases, and the transverse measure is equal to the Euclidean measure of the shorter base.

\begin{lem}
	\label{lem:trapezoid-foliation}
	Let $\mathcal{Q}$ be a Euclidean trapezoid whose bases are vertical segments of lengths $\ell$ and $\ell_1$, $\ell <\ell_1\leq L_1<\infty$, and one side is orthogonal to both bases with length at least $L>0$. Then there exists a partial measured lamination of $\mathcal{Q}$ whose leaves are Euclidean lines connecting the two base sides and whose transverse measure on one base side is the Euclidean measure and on the other base side is a multiple of the Euclidean measure. Moreover, the Dirichlet integral is, at most, a constant multiple of the length $\ell$ of the shorter side, with the constant depending on $L>0$ and $L_1$.
\end{lem}

\begin{proof}
	We position $\mathcal{Q}$ to have one base on the positive $y$-axis starting at the origin and the other base starting at point $L>0$ on the $x$-axis and orthogonal to the $x$-axis while staying in the upper half-plane. Let $0\leq t\leq 1$. We connect the point $t\ell$ on the left base side of $\mathcal{Q}$ to the point at height $t\ell_1$ on the right base side by Euclidean line. The equation of the line is
	$$
	y-t\ell=\frac{t\ell_1-t\ell }{L}x.
	$$
	From this formula, we get $t=y/(\frac{\ell_1-\ell}{L}x+\ell )$. Define
	$$
	v(x,y)=\ell t=\frac{\ell y}{\frac{\ell_1-\ell}{L}x+\ell}.
	$$
	A straightforward differentiation and integration gives
	$$
	\mathcal{D}_\mathcal{Q}(v)=\frac{\ell^2L}{\ell_1-\ell}\log\frac{\ell_1}{\ell}+\frac{\ell^2(\ell_1-\ell )}{3L}\log\frac{\ell_1}{\ell}.
	$$
	From the inequality, 
	$$
	\log \frac{\ell_1}{\ell}=\log \Big{[}1+\Big{(}\frac{\ell_1}{\ell}-1\Big{)}\Big{]}\leq \frac{\ell_1-\ell}{\ell}
	$$
	we get
	$$
	\mathcal{D}_\mathcal{Q}(v)\leq L\ell +\frac{(\ell_1 -\ell)^2}{3L}\ell =(L +\frac{L_1^2}{3L})\ell .
	$$
\end{proof}

\noindent {\it Proof of Theorem \ref{thm:int-upper-bdd}.}
We first define foliations $\mathcal{F}_n$ in every standard collar $\mathcal{C}(\alpha_n)$ as in Lemma \ref{lem:collar-foliation}. The transverse measures of $\mathcal{F}_n$ on all $I_n^i$ and $J_n^j$ are proportional to the arc lengths on these arcs, and their total masses are given by the weights on the edges of the train track corresponding to these arcs.

\begin{figure}[h]
\leavevmode \SetLabels
\endSetLabels
\begin{center}
\AffixLabels{\centerline{\epsfig{file=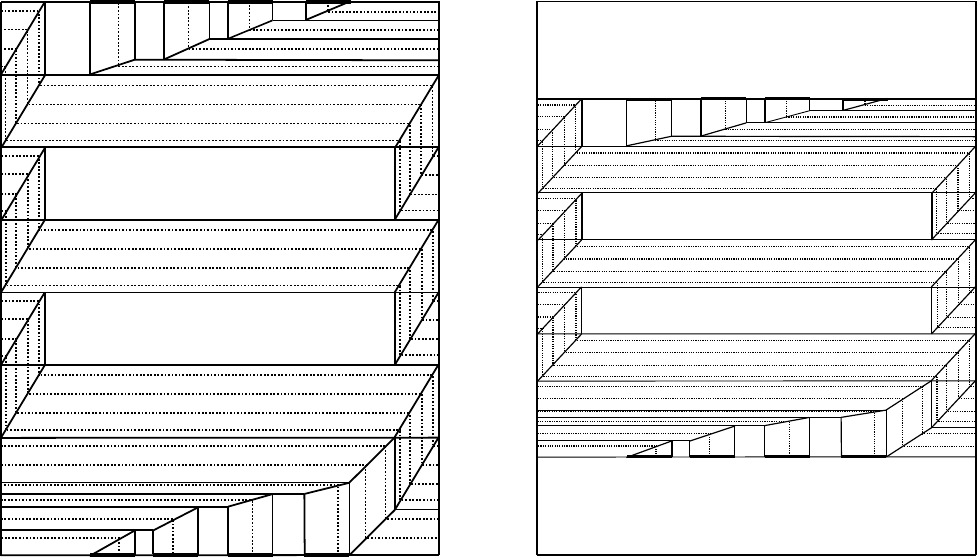, scale=0.75} }}
\vspace{-20pt}
\end{center}
\caption{The shrinking of the collar.} 
\label{fig:shrinking}
\end{figure}

For each subarc $g_{n,m}$ of the orthogeodesic that connects $\mathcal{C}(\alpha_n)$ and $\mathcal{C}(\alpha_m)$, we define the measured foliation $\mathcal{F}_{n,m}$ as in Lemma \ref{lem:connector}. The leaves of $\mathcal{F}_{n,m}$ are hypercycles of $g_{n,m}$, i.e., sets of points equidistant to $g_{n,m}$. The support of the foliation $\mathcal{F}_{n,m}$ is the set of points on distance at most $$d_{n,m}=\frac{q}{2}\min\{\ell_Y(\alpha_n),\ell_Y(\alpha_m)\}$$ from $g_{n,m}$. The transverse measure is proportional to the transverse distance between the equidistant lines to $g_{n,m}$ (see Lemma \ref{lem:connector}). Therefore, the transverse measure of $\mathcal{F}_{n,m}$ on the subarcs of $\partial_1\mathcal{C}(\alpha_n)$ and $\partial_2\mathcal{C}(\alpha_m)$ are proportional to the arc lengths on these arcs (because it is proportional to the distances on both $\alpha_n$ and $\alpha_m$ and then use Lemma \ref{lem:bdry_collar_projection}). 

The piecewise-defined measured foliations do not glue properly on the boundaries $\partial_i\mathcal{C}(\alpha_n)$, even though the intervals where the leaves meet are centered at the intersections of $\partial_i\mathcal{C}(\alpha_n)$ with the orthogeodesics. This is so because the choice of the widths of the neighborhoods of $g_{n,m}$ is given by $q$ times the minimum of the lengths of the two cuffs $\alpha_n$ and $\alpha_m$. Since the lengths of the cuffs may go to zero, we cannot ensure that the minimum is equal to both lengths even under a quasiconformal map of the whole surface. Therefore, we need to adjust the sizes of contact arcs between the collars and neighborhoods of the orthogeodesics.

We will adjust the sizes of $I_n^i$ and $J_n^j$ to match those of the arcs corresponding to the foliation of the neighborhoods of the orthogeodesics. 
We map $\widetilde{\mathcal{C}}(\alpha_n)$ by $z\mapsto\log z$ onto  the rectangle $Q_n=[0,\ell (\alpha_n)]\times [\frac{\pi }{2}-\theta ,\frac{\pi}{2}+\theta ]$, where $0<\theta <\frac{\pi}{2}$. The arcs $I_n^i$ and $J_n^j$ are mapped onto the top and bottom sides of the rectangle, and the proportion of their Euclidean sizes is the same as the proportions in the hyperbolic lengths of the subarcs to the whole boundary arc. In fact, the transverse measure becomes simply the Euclidean measure.

Let 
$$f_{\epsilon}(x,y)=(x,(1-\epsilon )(y-\frac{\pi}{2})+\frac{\pi}{2})
$$
be the map which shrinks $Q_n$ by the factor $(1-\epsilon )$ with $\epsilon >0$ (see Figure \ref{fig:shrinking}). 
If $\ell_Y(\alpha_n)\leq \ell_Y(\alpha_m)$ then $d_{n,m}=\frac{q}{2}\ell_Y(\alpha_n)$ and the arcs from $R_{n,m}$ and $\mathcal{C}(\alpha_n)$ are equal. In this case, we can complete the connection between $R_{n,m}$ and $f_{\epsilon}(\mathcal{C}(\alpha_n))$ by two rectangles with the foliation being the vertical lines and the Dirichlet integral at most $\frac{2\epsilon\theta}{q}\frac{[w(e_{n,m})]^2}{\ell_Y(\alpha_n)}$. 

The other possibility is that the neighborhood width $d_{n,m}$ around $g_{n,m}$ is strictly smaller than the length of the corresponding side $I_n^i$ which happens when $\ell_Y(\alpha_m)<\ell_Y(\alpha_n)$. We will consider the situation directly in the logarithmic coordinates. 
The above quasiconformal map $f_{\epsilon}$ leaves the vertical space of size $\epsilon\theta $ between the intervals on the top of the rectangle $Q_n$ whose lengths are ${q}\ell (\alpha_m)$ which are strictly smaller than the lengths $q\ell (\alpha_n)$ of the intervals coming from $f_{\epsilon}(\widehat I_n^i)$ (see Figure \ref{fig:small-trapezoid}). Each two intervals subtend two trapezoids $T_{n,m}^1$ and $T_{n,m}^2$, and we apply Lemma \ref{lem:trapezoid-foliation} to construct a partial foliation between the two intervals inside the trapezoids which we denote by $\mathcal{T}_{n,m}^1$ and $\mathcal{T}_{n,m}^2$. The partial foliation $v$ on each trapezoid has the transverse measure equal to the Euclidean measure on the shorter interval times a constant.

\begin{figure}[h]
\leavevmode \SetLabels
\endSetLabels
\begin{center}
\AffixLabels{\centerline{\epsfig{file=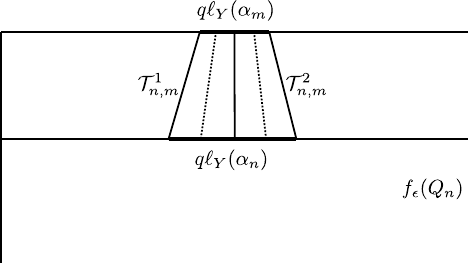, scale=0.85} }}
\vspace{-20pt}
\end{center}
\caption{Connecting different sizes.} 
\label{fig:small-trapezoid}
\end{figure}

Then, for each such $v$ that defines the foliation in the trapezoids, adjust the transverse measure by $\frac{2w(e_{n,m})}{q\ell (\alpha_m)}v(x,y)$. Lemma \ref{lem:trapezoid-foliation} gives
$$
\mathcal{D}(\frac{2w(e_{n,m})}{q\ell (\alpha_m)}v)\leq const\frac{w(e_{n,m})^2}{\ell (\alpha_m)}\leq const \frac{[i(\alpha_m,\mu ) ]^2}{\ell_Y(\alpha_m)}.
$$

We define $\mathcal{F}_{n,m}$ as in Lemma \ref{lem:connector} for each $g_{n,m}$ in each pair of pants. Then we construct $f_{\epsilon}(\mathcal{F}_n)$ for each collar $\mathcal{C}(\alpha_n)$ and note that $\mathcal{D}_{\mathcal{C}(\alpha_n)}(f_{\epsilon}(\mathcal{F}_n))\leq \frac{1}{1-\epsilon} \mathcal{D}_{\mathcal{C}(\alpha_n)}(\mathcal{F}_n)$ because the quasiconformal constant of $f_{\epsilon}$ is $\frac{1}{1-\epsilon}$. Finally, we draw the trapezoids connecting $\mathcal{F}_{n,m}$ and $f_{\epsilon}(\mathcal{F}_n)$. The transverse measures on the connecting places agree, and the constructed foliation $\mathcal{F}$ is homotopic to $\mu$. 

Finally, each $g_{n,m}$ gives a single part of the foliation $\mathcal{F}$, each foot of each orthogonal induces two trapezoids or two rectangles, and each collar induces a single part of the foliation $\mathcal{F}$. Therefore
\begin{align*}
\mathcal{D}_{Y}(\mathcal{F})
&\leq \sum_n\mathcal{D}(f_{\epsilon}(\mathcal{F}_n)+\sum_{n,m}\mathcal{D}(\mathcal{F}_{n,m})+16\sum_{n,m}\mathcal{D}(\mathcal{T}_{n,m}^1)\\
&\leq const \sum_{n=1}^{\infty}\Big{\{} \frac{[i(\mu ,\alpha_n)]^2}{\ell_X(\alpha_n)}+ {\ell_X(\alpha_n)}{[i(\mu ,\beta_n)]^2}\Big{\}}<\infty.
\end{align*}
$\hfill\qed$

\section{Random walks and the type problem}
\label{sec:rw-type}

In this section, we prove Theorem \ref{thm:graph}. The Hopf-Tsuji-Sullivan theorem states that the geodesic flow on the unit tangent bundle $T^1X$ of a Riemann surface $X$ is ergodic if and only if the Brownian motion on $X$ is recurrent if and only if the Poincar\' e series of the Fuchsian group $\Gamma$ of $X=\mathbb{H}/\Gamma$ is divergent if and only if $X$ does not have a Green's function. It is well-known that the above conditions are satisfied when $X$ is a compact (genus at least two) or a (hyperbolic) finite-area Riemann surface, and the above properties are not satisfied when the Fuchsian group $\Gamma$ is of the second kind. When the Fuchsian group $\Gamma$ is of the first kind but not finitely generated, then the Riemann surface $X=\mathbb{H}/\Gamma$ can either have ergodic geodesic flow or it may not. The type problem is deciding when an explicitly given Riemann surface has ergodic geodesic flow. The classical notation for this class is $X\in O_G$, which should be read as the Riemann surface $X$ has no Green's function.

We are interested in determining when a Riemann surface given by the Fenchel-Nielsen parameters on a pants decomposition lies in $O_G$. The standing assumption is that the cuff lengths are bounded above. In this case, the deformations of the surface $X$ by twists along the cuffs, which are at most the full length of the cuffs, are quasiconformally equivalent to the original Riemann surface structure. Since the class $O_G$ is invariant under the quasiconformal deformations, our conditions will only contain the lengths of the cuffs.

We now prove Theorem \ref{thm:graph} from Theorem \ref{thm:realization}.




\begin{proof}[Proof of Theorem \ref{thm:graph}]
Let $\{\alpha_n\}_{n\in\mathbb{N}}$ be the cuffs of the pants decomposition. By assumption, for some $C>0$, for all $n\in\mathbb{N}$, we have
$
\ell_X(\alpha_n)\leq C.
$

We first assume that the random walk on $\mathcal{G}=(V,E)$ is transient. We will construct a finite-area holomorphic quadratic differential on $X$ whose trajectories are transient. Once this is accomplished, Theorem 1.1 from \cite{Saric23} implies that $X\notin O_G$. 

Since the graph $\mathcal{G}=(V,E)$ may not be simple, we replace it by a simple graph $\hat{\mathcal{G}}=(V,\hat{E})$ with the same set of vertices $V$ and identifying edges connecting the same vertices. Then the conductance on an edge $\hat{e}\in \hat{E}$ obtained by identifying edges $e_1,e_2\in E$ is the sum of their conductances. By the parallel law, the random walk on $\mathcal{G}=(V,E)$ is transient if and only if the random walk on $\hat{\mathcal{G}}=(V,\hat{E})$ is transient (see \cite[page 28, II]{LyonsPeres}).

We choose an arbitrary orientation of each edge $\hat{e}\in \hat{E}$. 
Denote by $\hat{e}^-$ and $\hat{e}^+$ the initial and the end point of $\hat{e}$. In the language of Reversible Markov Chains as in \cite[I.2.A]{Woess}, we define the {\it conductance} of an edge $\hat{e}$ obtained by identifying $e_1,e_2\in E$ to be $$\hat{a}(\hat{e})=\ell_X(\alpha_{e_1})+\ell_X(\alpha_{e_2}),$$ where $\alpha_{e_i}$ is the cuff corresponding to the edge $e_i\in E$. If $\hat{e}\in \hat{E}$ corresponds to one edge $e\in E$ then the conductance is
$$
\hat{a}(\hat{e})=\ell_X(\alpha_e).$$

The {\it total conductance} at $x\in V$ is $$\hat{a}(x)=\sum_{\hat{e}\in x}a(\hat{e}),$$ where the sum is over all edges $\hat{e}$ that have at least one endpoint $x$.
Define the {\it resistance} $$\hat{r}(\hat{e})=\frac{1}{\hat{a}(\hat{e})}$$ along edge $\hat{e}\in\hat{E}$. 
Let $\ell^2(V,\hat{a}(\cdot ))$ and $\ell^2(\hat{E},\hat{r}(\cdot ,\cdot ))$ be the spaces of $\ell^2$-functions on the sets of vertices and edges of $\hat{\mathcal{G}}$, where the first space is weighted by the total conductance $\hat{a}(x)$ and the second space is weighted by the resistance $\hat{r}(\hat{e})=\frac{1}{\hat{a}(\hat{e})}$. 

Define
$$
\nabla :\ell^2(V,\hat{a})\to\ell^2(\hat{E},\hat{r})
$$
by
$$
\nabla (f)(\hat{e})=\frac{f(\hat{e}^+)-f(\hat{e}^-)}{\hat{r}(\hat{e})},
$$
where $\hat{e}^+$ is the end point of the edge $\hat{e}$ and $\hat{e}^-$ is the initial point of the edge $\hat{e}$ for the given orientation of $\hat{e}$.  
The adjoint map
$$\nabla^*:\ell^2(\hat{E},\hat{r})\to\ell^2(V,\hat{a})$$ is given by
$$
\nabla^*(g)(x)=\frac{1}{\hat{a}(x)}\Big{[}\sum_{\hat{e}:\hat{e}^+=x}g(\hat{e})-\sum_{\hat{e}:\hat{e}^-=x}g(\hat{e})\Big{]}. 
$$

By \cite[page 19, Theorem (2.12)]{Woess}, the random walk on $\hat{\mathcal{G}}$ is transient if and only if there exists a function $\hat{u}\in \ell^2(\hat{E},\hat{r})$ not identically equal to zero such that
\begin{equation}
\label{eq:transient_walk}
\nabla^*\hat{u}(y)=-\frac{i_0}{\hat{a}(x)}\delta_x(y),
\end{equation}
where $\delta_x(\cdot )$ is the Dirac measure with support $\{ x\}$ and $i_0$ is a non-zero constant. 

The existence of the flow function $\hat{u}$ depended on the fact that $\hat{\mathcal{G}}$ is a simple graph. We define a flow function $u\in\ell^2(E,r)$ for the non-simple graph $\mathcal{G}$ with $r(e)=\frac{1}{\ell_X(\alpha_e)}=\frac{1}{a(e)}$ for $e\in E$ using the flow function $\hat{u}$. For $e=\hat{e}$, set the orientation of $e$ to agree with the orientation of $\hat{e}$ and  $u(e)=\hat{u}(\hat{e})$. For $e_1,e_2\in E$ identified to an edge $\hat{e}\in\hat{E}$, set the orientations of $e_1,e_2$ to agree with the orientation of $\hat{e}$. 
Without loss of generality, assume that $\ell_X(\alpha_{e_1})\geq \ell_X(\alpha_{e_2})$. We set $u(e_1)=\hat{u}(\hat{e})$ and $u(e_2)=0$. Then $u$ satisfies (\ref{eq:transient_walk}) by the choice of the edge orientations. It remains to prove that $u\in\ell^2(E,r)$. Indeed
$$
\ell_X(\alpha_{e_1})\leq \ell_X(\alpha_{e_1})+\ell_X(\alpha_{e_2})\leq 2 \ell_X(\alpha_{e_1})$$
implies that
$$
r(e_1)\leq 2 \hat{r}(\hat{e}).
$$
Since $u(e_2)=0$ by the definition, we are not concerned how large $r(e_2)$ is. 
The above inequality implies $\| u\|_{\ell^2(E,r)}\leq 2\| \hat{u}\|_{\ell^2(\hat{E},\hat{r})}<\infty$. 

We use the obtained flow function $u$ on the non-simple graph $\mathcal{G}$ to show that $X\notin O_G$. 
We fix $x$ and a function $u$ as above. Then at each vertex $y$ different from $x$, $\sum_{e^+=y}u(e)-\sum_{e^-=y}u(e)=0$, For each oriented edge $e$, if $u(e)\geq 0$ we keep the orientation of $e$, and if $u(e)<0$ then we reverse the orientation of $e$. We replace $u$ with $|u|$ and note that $\sum_{e^+=y}|u|(e)-\sum_{e^-=y}|u|(e)=0$. By the abuse of notation, we keep the notation $u$ for the absolute value of $u$. 

We obtained a non-negative function $u\in\ell^2(E,r)$ that satisfies
$\nabla^*u(y)=-\frac{i_0}{a(x)}\delta_x(y)$ for a fixed vertex $x$. Let $P_x$ be the pair of pants that corresponds to $x$. We choose $i_0=a(x)$ and obtain a non-negative, non-trivial function $u\in\ell^2(E,r)$ such that
$$
\nabla^*u(y)=-\delta_x(y).
$$
The above condition implies that at each $y\neq x$, the total $u$-weight of the edges $e$ with $e^+=y$ is equal to the total $u$-weight of the edges $e$ with $e^-=y$.

We use the function $u$ and the oriented edge set $E$ to define a train track $\tau_u$ on $X$ with weights satisfying the switch conditions on the vertices of the train track. The train track $\tau_u$ will have its vertices on the cuffs of the pants decomposition, dividing each cuff into finitely many edges, and additional edges will connect the cuffs in pairs of pants. We first specify the direction of the tangency of the edges connecting the cuffs. Let $\alpha$ be a cuff on the boundary of $P_y$ and $P_z$. Given a segment crossing from $P_y$ to $P_z$ and intersecting $\alpha$ at one point, we choose tangencies on both sides of $\alpha$ to be from the right as seen from the segment (see Figure \ref{fig:tangencies}). We form a train track $\tau_u$ by connecting the cuffs by orthogeodesics and smoothing at the feet to the right (as in Figure \ref{fig:tangencies}).

\begin{figure}[h]
\leavevmode \SetLabels
\endSetLabels
\begin{center}
\AffixLabels{\centerline{\epsfig{file=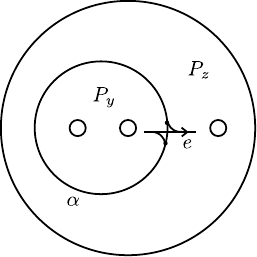, scale=0.9}}}
\vspace{-20pt}
\end{center}
\caption{The tangencies of $\tau_u$ at a cuff obtained from the orientation of edges.} 
\label{fig:tangencies}
\end{figure}

\begin{figure}[h]
\leavevmode \SetLabels
\endSetLabels
\begin{center}
\AffixLabels{\centerline{\epsfig{file=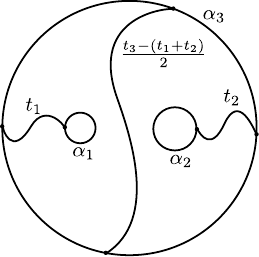, scale=0.9} }}
\vspace{-20pt}
\end{center}
\caption{The train track $\tau_u$ on $P_x$ when $t_1+t_2\leq t_3$ with two different relative positions of tangencies.} 
\label{fig:notT}
\end{figure}

Let $e_1$, $e_2$, and $e_3$ be the edges of the graph $\mathcal{G}$ whose one vertex is $x$, where we have 
$\nabla^*u(x)=-1.$ Denote by $\alpha_i$ for $i=1,2,3$ the cuffs of $P_x$ that correspond to $e_i$.  Recall that the edges are oriented and that the function $u$ is non-negative. For simplicity, write $t_i=u(e_i)$ for $i=1,2,3$. If we have
$$
t_1+t_2\leq t_3,
$$
then we form a train track in $P_x$ with the weights on the edges connecting the cuffs as in Figure \ref{fig:notT}. Similarly, if the weights $t_i$ for $i=1,2,3$ fail to satisfy another triangle inequality, we can form the corresponding train track.

The remaining possibility is that the weights satisfy all three possible triangle inequalities. Then we obtain the train track as in Figure \ref{fig:T}. The weights are given by
\begin{equation*}
x=\frac{1}{2}(t_1+t_3-t_2),\
y=\frac{1}{2}(t_1+t_2-t_3),\ \mathrm{and}\ 
z=\frac{1}{2}(t_2+t_3-t_1).
\end{equation*}

\begin{figure}[h]
\leavevmode \SetLabels
\endSetLabels
\begin{center}
\AffixLabels{\centerline{\epsfig{file=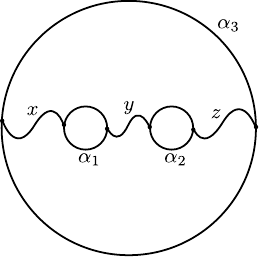, scale=0.9} }}
\vspace{-20pt}
\end{center}
\caption{The train track $\tau_u$ on $P_x$ when all three triangle inequalities hold.} 
\label{fig:T}
\end{figure}

If $x$ meets only two edges $e_1$ and $e_2$, then we have $t_1\leq t_2$ or the opposite inequality. This means that $P_x$ has two cuffs and a puncture. The train track $\tau_u$ is given in Figure \ref{fig:punctureT}. Note that in all cases the total sum of the weights at each cuff is equal to the value of the function $u$ on the edge corresponding to the cuff.
\begin{figure}[h]
\leavevmode \SetLabels
\endSetLabels
\begin{center}
\AffixLabels{\centerline{\epsfig{file=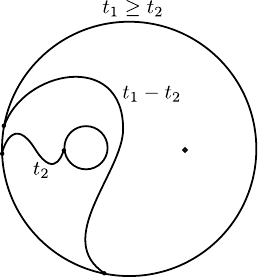, scale=0.9} }}
\vspace{-20pt}
\end{center}
\caption{The train track $\tau_u$ on $P_x$ with one puncture.} 
\label{fig:punctureT}
\end{figure}

Next, we consider the train track $\tau_u$ on the pair of pants $P_y$ corresponding to a vertex $y$ different from $x$ of the edges with one vertex at $x$. Let $\alpha_i$ for $i=1,2,3$ be the cuffs of $P_y$, and let $t_i=u(e_i)$ for the edges $e_i$ corresponding to $\alpha_i$. Then we have $t_1+t_2=t_3$ or an analogous equality because $\nabla^*u(y)=0$. The train track with appropriate weights is given in Figure \ref{fig:attachPy}, and we note that the sum of the weights of the edges at a cuff is again equal to the value of the function $u$ on the edge corresponding to the cuff. The case when one end is a puncture is left to the reader.

\begin{figure}[h]
\leavevmode \SetLabels
\endSetLabels
\begin{center}
\AffixLabels{\centerline{\epsfig{file=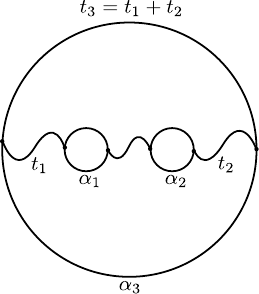, scale=0.9} }}
\vspace{-20pt}
\end{center}
\caption{The train track $\tau_u$ on $P_y$.} 
\label{fig:attachPy}
\end{figure}

We formed a train track $\tau_u$ in this process, together with the appropriate weights $w$ on the edges connecting the cuffs. The sum of the weights from one side of a cuff equals the sum of the weights on the other side of the cuff because both sums are the value of the function $u$ on the edge of $E$ representing the cuff. We define the weights on the edges of the train track that lie on the cuffs such that the switch condition at the vertices of the train track holds (the sum of the weights tangent in one direction equals the sum of the weights tangent in the other direction). We establish this extension using the smallest possible edge weights on the train track that lie on the cuffs. This is standard and is illustrated in Figure \ref{fig:weights-cuff}, and the statement follows since there are only finitely many positions.

\begin{figure}[h]
\leavevmode \SetLabels
\endSetLabels
\begin{center}
\AffixLabels{\centerline{\epsfig{file =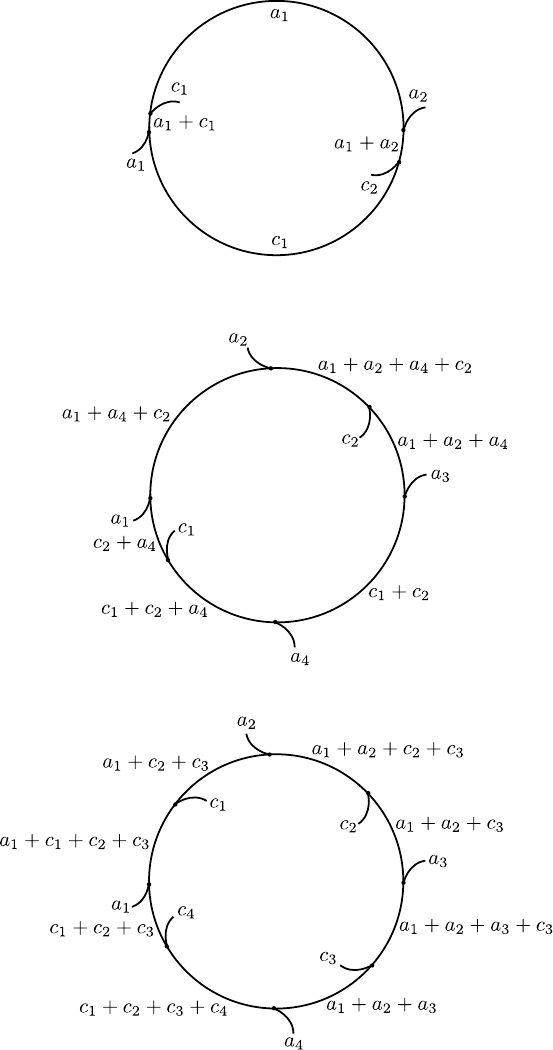,scale=.9} }}
\vspace{-20pt}
\end{center}
\caption{The weights on cuffs.} 
\label{fig:weights-cuff}
\end{figure}

Let $E(\tau_u )$ be the set of edges of the constructed train track $\tau_u$. 
Therefore, we obtained an assignment of weights $w:E(\tau_u )\to\mathbb{R}$ that is non-negative and satisfies switch conditions. This defines a measured lamination on $X$, denoted by $\mu_u$ (see \cite{Saric20} and \cite[Expos\' e One]{FLP}). 

For any edge $g$ of the constructed train track $\tau_u$, denote by $w(g)$ the assigned weight. 
If $g$ connects two cuffs of a pair of pants $P$, then $$w(g)\leq \sum_j u(e_j),$$ where $e_j$ are edges of $\mathcal{G}$ ending at the vertex corresponding to $P$. If $g$ is an edge of the train track $\tau_u$ on a cuff $\alpha$ of the pants decomposition, then $$w(g)\leq 2u(e_{\alpha}),$$ where $e_{\alpha}$ is the edge of $\mathcal{G}$ corresponding to the cuff $\alpha$ (see Figure \ref{fig:weights-cuff}). 

We gather the above estimates to prove that $\mu_u$ is obtained by straightening the leaves of a finite-area holomorphic quadratic differential on $X$. Let $\{\alpha_n\}$ be the cuffs of the pants decomposition and $\{\beta_n\}$ the transverse family of simple closed geodesics as in Theorem \ref{thm:int-upper-bdd}.

Note that $i(\mu_u ,\alpha_n)\leq u(e_{\alpha_n})$ by the definition of $\mu_u$. Since $i(\mu_u ,\beta_n)\leq\max_{g} w(g)$ where the supremum is over all edges $g$ of the train track $\tau$ that lie on $\alpha_n$, the above estimate gives
$$
i(\mu_u ,\beta_n)\leq 2 u(e_{\alpha_n}).
$$ 
Then we have
\begin{align}
 &\quad \sum_{n=1}^{\infty}\Big{\{} \frac{[i(\mu_u ,\alpha_n)]^2}{\ell_X(\alpha_n)}+ {\ell_X(\alpha_n)}{[i(\mu_u ,\beta_n)]^2}\Big{\}}\\ \nonumber
 &\leq C \sum_{n=1}^{\infty}\Big{\{} \frac{u(e_{\alpha_n})^2}{\ell_X(\alpha_n)}+ {\ell_X(\alpha_n)}{u(e_{\alpha_n})^2}\Big{\}}\\ \nonumber
 &\leq C_1 \sum_{n=1}^{\infty}{u(e_{\alpha_n})^2}{r(e_{\alpha_n})}<\infty. \nonumber
 \end{align}

It remains to prove that there is a set of transient leaves of $\mu_u$ of positive transverse measure. By \cite[Theorem 1.6]{Saric23}, there exists a finite-area holomorphic quadratic differential $\varphi_u$ on $X$ whose horizontal measured foliation is homotopic to $\mu_u$. We form an increasing exhaustion of $X\setminus P_x$ by finite area surfaces $\{ X_n\}_n$ that are unions of finitely many pairs of pants at each finite stage $X_n$. In addition, the boundary of each $X_n$ contains all boundary cuffs of $P_x$. We replace each boundary cuff with a step curve in a natural parameter of $\varphi$ (see \cite{Strebel}) within $1/100$ of the cuff's collar width. We consider all rays of the horizontal foliation in $X_n$ that start on the step curves corresponding to the boundary cuffs of $P_x$. If a ray hits the same boundary component at its other endpoint and can be homotoped to this boundary component, then we erase it. The remaining rays are grouped into two sets: the rays that are homotoped to the edges of the graph $\mathcal{G}$ that have their initial endpoints at the vertex $x$ (corresponding to pants $P_x$), and the rays that are homotoped to the edges of $\mathcal{G}$ that have their terminal endpoints at $x$. A ray can start and end at a cuff of $P_x$. However, the transverse measure of rays that start at a cuff of $P_x$ is larger by $1$ than the transverse measure of the rays that end at a cuff of $P_x$ by $\nabla^*u(x)=-1$. 

Consider the set of rays in each $X_n$ corresponding to the edges of $\mathcal{G}$ with the initial point at $x$. We define a flow on this set by orienting each ray so that it starts at the cuff. Then we have a well-defined flow on this set of rays that preserves $|\varphi_u|$-area. The Poincar\'e recurrence theorem implies that almost all rays hit the boundary of $X_n$ in the positive direction.  Since the transverse measure of the ray starting at the cuffs of $P_x$ is at least $1$ plus the transverse measure of the ray returning to the cuffs of $P_x$, it follows that the set of rays of transverse measure at least $1$ leaves the cuffs of $P_x$ and hits other parts of the boundary of $X_n$. By letting $n\to\infty$, we conclude that there is a set of rays of transverse measure at least $1$ starting at the cuffs of $P_x$ and going off to infinity. Therefore, the measured lamination $\mu_u$ has leaves of positive transverse measure that go to infinity of $X$, and we conclude that $X\notin O_G$.

\vskip .2 cm

To prove the opposite direction, assume that $X\notin O_G$. In \cite[Theorem 4.2]{Saric23}, a finite area holomorphic quadratic differential is constructed so that almost all the leaves connect a single geodesic pair of pants with infinity. We only consider the first part of this construction. Namely, let $P_x$ be the fixed pair of pants of the pants decomposition that corresponds to a fixed vertex $x\in V$ of the graph $\mathcal{G}$. Then there exists a harmonic function $g:X\setminus P_x$ such that $g|_{\partial P_x}=0$ and $0< g(z)<1$ for all $z\in X\setminus \bar{P}_x$ (see \cite{AhlforsSario}). Let $g^*(z)$ denote the local harmonic conjugate of $g$. Then $d(g+ig^*)$ is an Abelian differential on $X\setminus \bar{P}_x$, $[d(g+ig^*)]^2$ is integrable on $X\setminus P_x$,  and its horizontal foliation ${H}_g$ consists of curves $g^*(z)=const$. The function $g(z)$ monotonically increases along each horizontal trajectory starting at $\partial P_x$. This gives a global orientation on the horizontal leaves of ${H}_g$ that start at $\partial P_x$ and leave every finite area subset of $X\setminus P_x$ (we remove recurrent leaves). The transverse measure to the horizontal foliation ${H}_g$ is given by integrating $dg^*$ along transverse arcs. Let $\mu_g$ be a measured lamination obtained by extending ${H}_g$ to a proper partial measured foliation on $X$, straightening the leaves, and pushing forward the measure. By \cite[Theorem 4.2]{Saric23}, $\mu_g\in ML_{int}(X)$.  The leaves of the support of $\mu_g$ in $X\setminus P_x$ inherit the orientation from the leaves of the horizontal foliation of ${H}_g$.

We first choose an arbitrary orientation of all edges of $\mathcal{G}$. An edge $e$ connects $e^-=y_1$ to $e^+=y_2$, where $y_1$ and $y_2$ are pairs of pants. Let $\alpha$ be the cuff on the common boundary of $P_{y_1}$ and $P_{y_2}$. We define $u(e)$ to be the transverse measure of the geodesics of the support of $\mu_g$ that intersect $\alpha$ and have the same orientation as $e$ minus the transverse measure of the geodesics of the support of $\mu_g$ that intersect $\alpha$ and have opposite orientation to $e$. For every vertex $y\neq x$, the total transverse measure of the geodesics of the support of $\mu_g$ that enter the corresponding pair of pants is equal to the transverse measure of the geodesics exiting it. Thus, we have that (\ref{eq:transient_walk}) holds for $y\neq x$, and it remains to consider the edges with one endpoint at $x$. Orient all edges to start at $x$ and define $u$ on each such edge to be the transverse measure of the geodesics of the support of $\mu_g$ that intersect corresponding arcs. Then (\ref{eq:transient_walk}) holds for $y= x$ with an appropriate constant. 
Since $\mu_g$ is in $ML_{int}(X)$ it satisfies 
$$
\sum_n \Big{[}\frac{i(\alpha_n,\mu_g )^2}{\ell_X(\alpha_n)}+i(\beta_n,\mu_g )^2\ell_X(\alpha_n)\Big{]} <\infty.
$$
Since $|u(e_n)|\leq i(\alpha_n ,\mu_g )$ when $\alpha_n$ is the cuff corresponding to $e_n$, we have
$u\in\ell^2(E, r)$ by $\sum_n \frac{i(\alpha_n,\mu_g )^2}{\ell_X(\alpha_n)} =\sum_n {[u(e_n) ]^2}{r(e_n)}<\infty$. Moreover, $u$ is not identically zero because $u(e)\neq 0$ for at least one edge $e$ corresponding to a cuff of $P_x$ by the definition of $g$ (the minimum of $g$ is achieved on $\partial P_x$). Therefore, $\mathcal{G}$ is transient. This finishes the proof.
\end{proof}

\section{Rough isometries for Riemann surfaces with upper-bounded pants decompositions}

\label{sec:rough}

Kanai \cite{Kanai} and Varopoulos \cite{Var} proved that if two Riemannian manifolds with upper-bounded Ricci curvatures and positive injectivity radius are roughly isometric then they both either do not support a Green's function or they both do. Their methods seem to depend on the injectivity radius being positive. However, there are many infinite Riemann surfaces with cuff lengths going to zero over a subsequence, which is not addressed by the above two papers. 

In this section, we prove Theorem \ref{thm:rough-isom} and Theorem \ref{thm:rough-isom} stated in introduction. We also give a short proof of Kanai-Varopoulos theorem. We first use Theorem \ref{thm:graph} to give an example of two roughly isometric Riemann surfaces with upper-bounded cuff lengths such that the above theorem of Kanai and Varopoulos is false.

\begin{thm}
    \label{thm:kanai_ce}
    There exist two roughly isometric Riemann surfaces with upper-bounded pants decompositions such that one has recurrent Brownian motion and the other has transient Brownian motion.
\end{thm}

\begin{proof}
Let $X$ be the Cantor tree surface with $\ell (\alpha_e)=\frac{1}{2^n}$ for all $e\in E_n$, where $\alpha_e$ the corresponding cuff and $n\in\mathbb{N}$. By Theorem \ref{thm:pandazis-extended}, $X\in O_G$. 
\begin{figure}[h]
 	\leavevmode \SetLabels
	\endSetLabels
\begin{center}
\AffixLabels{\centerline{\epsfig{file =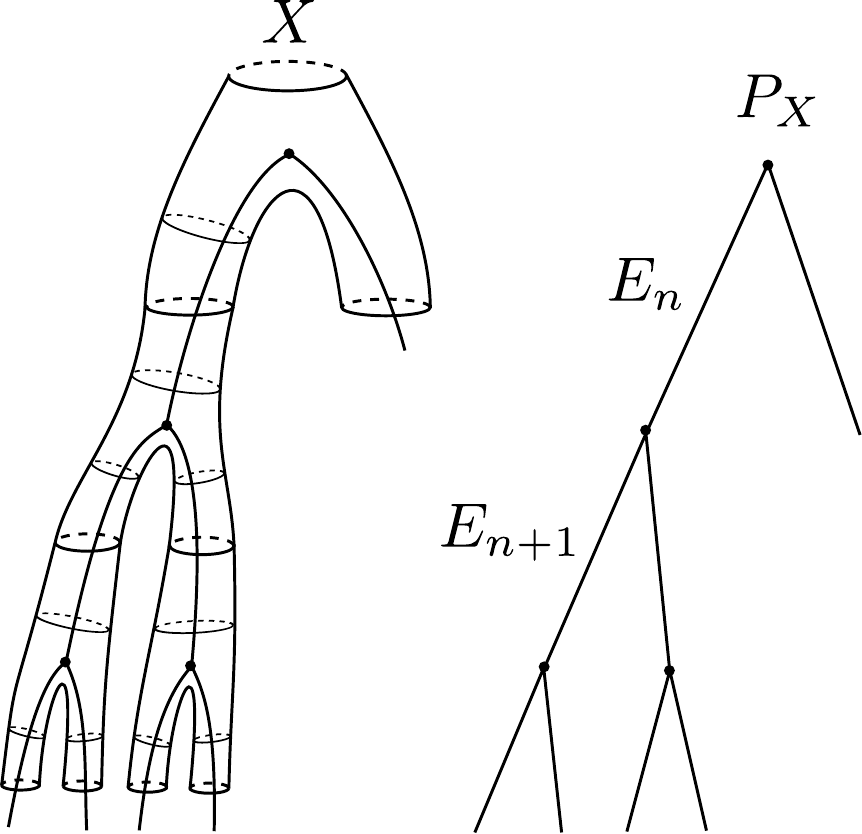, scale=.3} }}
\AffixLabels{\centerline{\epsfig{file =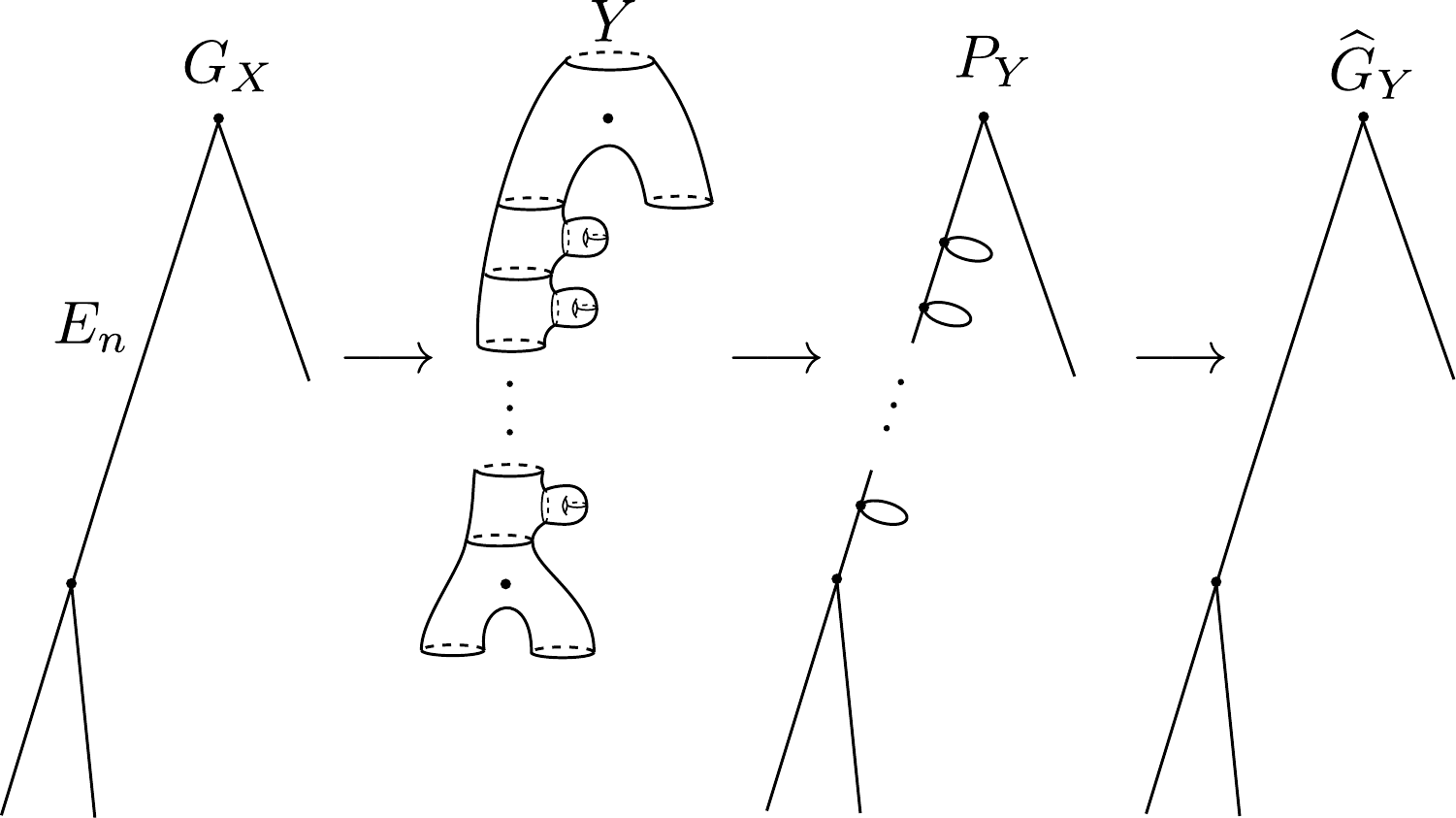, scale=.3} }}
 		\vspace{-20pt}
	\end{center}
 	\caption{From the surface $X$ to the surface $Y$.} 
    \label{fig:X-Y} 
\end{figure}

We form another Riemann surface $Y$ starting from the surface $X$. Note that the edges of the dual graph $G_X$ of $X$ in $E_n$ have resistance $2^n$. The standard collar on $X$ around $\alpha_e$ for $e\in E_n$ has width of the order $\log 1/2^{-n}\asymp n$. We replace each edge $e\in E_n$ by a chain of $n$ two holed tori that are glued linearly and have cuff lengths constantly equal to $1$. Each vertex is replaced by a pair of pants also with cuff lenghts $1$ (see Figure \ref{fig:X-Y}). The choice of twists is irrelevant since the existence of Green's function is a quasiconformal invariant.

We prove that $X$ and $Y$ are roughly isometric in the sense of Kanai \cite{Kanai}. For each pair of pants we fix a point in the complement of the standard collars and connect points of the adjacent pairs of pants by geodesic arcs (see Figure \ref{fig:X-Y}). The union of these arcs on $X$ and $Y$ form length spaces $P_X$ and $P_Y$, respectively. Because the lengths of the cuffs are bounded from the above, the length (metric) spaces $P_X$ and $P_Y$ are roughly isometric to $X$ and $Y$, respectively. Since the lengths of the edges of $P_X$ are asymptotically $\log 1/2^{-n}\asymp n$ and the number of two holed tori is $n$, the metric spaces $P_X$ and $P_Y$ are roughly isometric. Therefore, $X$ and $Y$ are roughly isometric. 
    
It remains to prove that $Y\notin O_G$. Consider the dual graph $G_Y$ to the pants decomposition of $Y$ as in Figure \ref{fig:X-Y}. The dual graph $G_Y$ is roughly isometric to $\widehat{G}_Y$ obtained by collapsing all loops. The recurrence/transience of the graph 
$\widehat{G}_Y$ is equivalent to that of the graph $\widetilde{G}_Y$, where $\widetilde{G}_Y$ is obtained by replacing the edges of $\widehat{G}_Y$ corresponding to $e\in E_n$ by a single edge whose resistance is the sum of the resistances (see  \cite[page 27, \S 2.3]{LyonsPeres}). Then $\widetilde{G}_Y$ is a rooted binary tree (same as $G_X$) but the resistance for an edge of $n$-th generation is a constant multiple of $n$. The random walk on such network is transient (in fact, it is ``very far'' from being recurrent as we need the resistance to be at least  $2^n/n$ for recurrence).   
\end{proof}

\begin{rem}
    The same idea of replacing a long collar with many twice holed tori for topological $\mathbb{Z}^d$-covers of a compact surface with $d\geq 3$ gives another example of two Riemann surfaces that are roughly isometric but one has recurrent Brownian motion and the other has transient.
\end{rem}

We are also able to give a short proof of Kanai and Varapoulos theorem for bounded pants decompositions.

\begin{thm}[Kanai \cite{Kanai} and Varopoulos \cite{Var}]
    \label{thm:kv}
    Let $X$ and $Y$ be two Riemann surfaces with bounded pants decompositions. If they are roughly isometric, then they are simultaneously in $O_G$ or not.
\end{thm}

\begin{proof}
    Recall that the metric spaces $P_X$ and $P_Y$ are roughly isometric to $X$ and $Y$, respectively. Moreover, $P_X$ and $P_Y$ are roughly isometric to the dual graphs $G_X$ and $G_Y$ (as metric spaces), respectively. By post-compositions with additional rough isometries, we arrange that the vertices of $G_X$ ($G_Y$) are mapped onto vertices of $P_X$ ($P_Y$). Since $X$ and $Y$ are roughly isometric by the assumption, it follows that $G_X$ and $G_Y$ are roughly isometric as weighted graphs. Then \cite[\S 2.6, Theorem 2.17]{LyonsPeres} implies the theorem.
\end{proof}

We conclude this section with the proof of Theorem \ref{thm:rough-isom2} stated in introduction.

\begin{proof}[Proof of Theorem \ref{thm:rough-isom2}]
From Theorem \ref{thm:graph}, it suffices to prove that the random walk of $\mathcal G$ is recurrent if and only if the random walk on $\mathcal G'$ is recurrent. Since   $\mathcal G$ and $\mathcal G'$ are roughly isometric, this is the content of the graph version of Kanai's Theorem,  see \cite[\S 2.6, Theorem 2.17]{LyonsPeres}.
\end{proof}

\section{Applications to the type problem}

\label{sec:typeex}

In this section, we discuss some applications of Theorem \ref{thm:graph}.

\subsection{The type problem on conductance graphs}

A conductance graph $G = (V,E,\ell)$ is a graph with a countable vertex set $V$ and a countable non-oriented edge set $E$ equipped with a conductance function $\ell : E \to [0,\infty)$ such that for any $x$, 
$$
\sum_{e \in x} \ell(e) < \infty,
$$
where the sum is over all edges $e$ adjacent to $x$. Loops and multi-edges are allowed in $G$. If $(x,y) \in V^2$, we denote by $\ell(x,y) = \ell(y,x) = \sum_{e \in x \cap y } \ell(e)$ where the sum is over all edges $e$ adjacent to $x$ and $y$.

The Markov chain $(X_t)_{t \in \mathbb N}$ on $V$ associated to a conductance graph $G$ has transition probability defined for $(x,y) \in V^2$ by
\begin{equation}\label{eq:defP}
P(x,y) = \frac{\ell(x,y)}{\sum_{z \in V} \ell(x,z)},
\end{equation}
with the convention that if $\sum_{z \in V} \ell(x,z) = 0$ then $P(x,x) = 1$. We say that $x$ is recurrent/transient if $(X_t)$ started at $X_0 = x$ is recurrent/transient.

Theorem \ref{thm:graph} reduces the study of the type problem of an infinite Riemann surface $X$ with an upper-bounded pants decomposition to the study of the type problem of its pant graph $\mathcal G$. 

There are multiple classical tools to determine transience or recurrence of the Markov chain \eqref{eq:defP} such as Royden criterion, Rayleigh's monotonicity,  the Nash-Williams criterion, network reduction or rough isometries. We refer to \cite[Chapter 2]{LyonsPeres} for a modern exposition. We will use two results which we quote explicitly: 

\begin{thm}[Rayleigh’s Monotonicity Principle] \label{thm:R} Let $G = (V,E,\ell)$ and $G' = (V,E,\ell')$ be two conductance graphs with the same base graph. Assume that for all $e \in E$, we have $\ell(e) \leq \ell'(e)$. For any $x \in V$, if $x$ is transient in $G$ then $x$ is transient in $G'$. Conversely, if $x$ is recurrent in $G'$ then $x$ is recurrent in $G$.
\end{thm}

Interestingly, Theorem \ref{thm:graph} allows to transfer this monotonicity for surfaces with a graph-isomorphic pant decomposition and monotone cuff-lengths.

The second classical tool that we shall use is the so-called Nash-Williams criterion which is a powerful tool to prove recurrence. Let $G= (V,E,\ell)$ and $x \in V$, we say that $\Pi \subset E$ is a cutset for $x$ if every infinite simple path from $x$ intersects $\Pi$. If $\Pi \subset E$, we set $\ell(\Pi) = \sum_{e \in \Pi} \ell(e)$.

\begin{thm}[Nash-Williams Criterion] \label{thm:NW} Let $G = (V,E,\ell)$ be an infinite  conductance graph and $x \in V$. Let $(\Pi_n)_{ n \in \mathbb N}$ be a sequence of pairwise disjoint cutsets for $x$. Then $x$ is recurrent if 
$$
\sum_{n=1}^\infty \ell(\Pi_n) ^{-1}= \infty. 
$$
\end{thm}

Establishing transience is usually more difficult. A classical method consists in defining a flow going to infinity, this is the content of Royden criterion,  see \cite{Lyons83,LyonsPeres} for a precise statement.

\subsection{The type problem for the Cantor tree surface}

Let $X$ be an infinite Riemann surface that is planar and whose space of ends is a Cantor set. We fix a topological pants decomposition of $X$ as in Figure \ref{fig:cantor}.

We denote by $V$ the set of pants, $E$ the set of cuffs. For integer $n \geq 1$, we denote by $E_n$ the set the cuffs of the pants decomposition that are boundaries on the level $n\geq 2$. We have $|E_1| = 1$ and $|E_n| = 2^n $ for $n \geq 2$ (where $| \cdot |$ denotes the cardinality of a set). The question that we are interested is to determine whether $X\in O_G$ or not from the data of the Fenchel-Nielsen parameters on the pants decomposition.

We are now ready to prove Theorem \ref{thm:pandazis-extended} stated in introduction.

\begin{proof}[Proof of Theorem \ref{thm:pandazis-extended}]
The last claim follows immediately by applying the first two claims with $\psi(n)$ equal to $C$ or $1/C$ times $\ell_X(E_n)$ (recall that $|E_n| = 2^n$).

	By Theorem \ref{thm:graph}, it is enough to determine whether the random walk $(X_t)_{t \in \mathbb N}$ on the graph of the Cantor tree surface $G = (V,E,\ell_X)$ is recurrent or transient. Moreover by Rayleigh's principle, Theorem \ref{thm:R}, it is sufficient to assume that $\ell_X(e) = \psi(n)/2^n$ for all $e \in E_n$. By transitivity, we can assume that $X_1 = x_1$ is a root vertex adjacent to $e \in E_1$.
    
Then the Nash-Williams criterion, Theorem \ref{thm:NW}, applied to $\Pi_n = E_n$ implies immediately that $X \in O_G$ if $\sum_{n \geq 1} (\psi(n))^{-1} = \infty$. 

Conversely, assume that $\sum_n \psi(n)^{-1} < \infty$. Let $D_t -1 \in \{0,1,2 , \cdots \}$ be the distance between $X_t$ and $x_1$. In order to prove the transience of $(X_t)$, it is sufficient to prove that the transience of $(D_t)$. We note that if $D_t = n $, the probability that $D_{t+1} = n+1$ given $(X_1,\ldots,X_t)$ is equal to 
$$
\frac{2 \psi(n+1)/2^{n+1}}{2 \psi(n+1) / 2^{n+1} + \psi(n)/2^n } = \frac{\psi(n+1)}{\psi(n+1) + \psi(n)} .
$$
It follows that $D_t$ is a Markov chain on $\mathbb N$ associated to the conductances $\ell(n,n+1) = \psi(n+1)$. It is well-known that this Markov chain is transient if and only if $\sum_n \psi(n)^{-1} < \infty$ (see for example \cite[XV.8]{Feller}). It concludes the proof. \end{proof}


We can also prove Theorem \ref{prop:shorten}  stated in introduction.

\begin{proof} [Proof of Theorem \ref{prop:shorten} ]
We start with the proof of (i). We apply  again the Nash-Williams criterion to $\Pi_n = E_n$. We have for $n$ large enough
$$
\ell_X( \Pi_n) = \sum_{e \in \Pi_n} \ell_X (e) \leq |E_n| \psi(n)/2^{n} + C |A_n| \leq 2 C K(n),
$$
where we have used that $\psi(n) \leq K(n)$ for all $n$ large enough. Since $\alpha \leq 1$, $\sum_n K(n)^{-1} = \infty$ and the conclusion follows from Theorem \ref{thm:NW}.

We next prove (iii) by another application of Nash-Williams. This time we cannot take $\Pi_n = E_n$ as cutset. Recall that a matching $M \subset E$ is a subset of edges  such that no two edges in $M$ share a common adjacent vertex in $V$.  We first claim that for any matching $M$ if $A_n  \subset M \cap E_n$ then $X \in O_G$.

First, since $\sum_{n} \psi(n)^{-1} = \infty$, either $\sum_{n} \psi(2n)^{-1} = \infty$ or $\sum_{n} \psi(2n+1)^{-1} = \infty$. Let us assume for example $\sum_{n} \psi(2n)^{-1} = \infty$.  We then define $\Pi_n = (E_{2n} \cup E_{2n+1}) \backslash  M$ (if $\sum_{n} \psi(2n+1)^{-1} = \infty$ we would have taken $\Pi_n = (E_{2n-1} \cup E_{2n}) \backslash  M$).

By construction $\Pi_n$ is a pairwise disjoint collection of cutsets. Indeed, any infinite simple path visits at least two adjacent edges in $E_{2n} \cup E_{2n+1}$. At least one of them is not in $M$ by the definition of matching, it is thus in $\Pi_n$. Moreover, since $\ell_X(e) \leq \psi(n)/2^n$ for all $e \in E_n \backslash A_n$, we find
$$
\ell_X (\Pi_n)   = \sum_{e \in \Pi_n} \ell_X (e)  \leq 2^{2n} \psi(2n) / 2^{2n} + 2^{2n+1} \psi(2n+1) / 2^{2n+1}\leq C \psi(2n),
$$
where we have used that $\psi(n+1)  \leq (C -1)\psi(n)$ for some $C >1$. Then $\sum_{n} \psi(2n)^{-1} = \infty$ and Theorem \ref{thm:NW} imply that $X \in O_G$.

To conclude the proof of (iii), it remains to prove that there exists a matching $M$ such that satisfies $|M_n| \sim (1/3) |E_n|$ where $M_n = M \cap E_n$. Indeed, then the claim follows by taking $A = M$. We say that a  matching is perfect if for all $v \in V$, there exists $e \in M$ which is adjacent to $v$. Let $\{ x_1,y_1\}$ be the vertices in $V$ adjacent to $ e_1 \in E_1$. We can construct iteratively by levels a perfect matching. For $n \geq 1$, let $V_n$ be the set of vertices at distance $n-1$ from $\{x_1,y_1\}$ in a perfect matching which are adjacent to an edge in $E_n$. We have $|M_n| = |M \cap E_n| = |V_n|$. For any $n \geq 1$, there are $2^{n}$ vertices at distance $n-1$ from $\{ x_1,y_1\}$. We get $|M_{n+1}| + |M_{n}| = 2^{n} $. Therefore $|M_{n}| \sim (1/3) 2^{n}  = (1/3) |E_n|$ for any perfect matching. Statement (iii) follows.

It remains to prove (ii). For some constant $\theta >0$ to be defined and any integer $k \geq 1$, we set $L_k = \theta 2^k / k^\alpha $. We build a subtree $T$ of the conductance graph of $X$ and define $A$ as the set of edges of $T$. This subtree is defined as follows: we pick a vertex $v_1$ at depth $n_1$  in the Cantor tree with $n_1$ to be fixed later. From each of the two offspring of $v_1$, we include in $T$ a line segment  of length $L_1$ and denote by, say  $v_{21},v_{22}$, the end vertices of the line segment. We then branch: we include in $T$ the two offspring of $v_{2i}$, $i = 1,2$  and we draw from there $2^2$ line segments of length $L_2$. We then branch again and repeat with $2^3$ line segments of length $L_3$ and so on. After $h$ steps, we have $2^h$ line segments and we have reached depth 
$$
n_{h+1} = n_1 + \sum_{k=1}^h L_h = n_1 + \theta \sum_{k=1}^h \frac{2^k}{k^\alpha}.
$$
For some $c_\alpha >0$, we have $n_{h+1} \geq n_1 + c_\alpha \theta 2^{h+3}/h^\alpha$. In particular, if $\theta$ and $n_1$ were chosen large enough, we have that $n_h (\log n_h)^{\alpha} \geq 2^{h}$ for all $h \geq 1$. Therefore, if $A$ is the set of edges of $T$, we have that $|A \cap E_n| \leq K(n)$ for all $n \geq 1$.

It remains to check that  if $\ell_X (e) = 1$ for all $e \in A$, then  $X \notin O_G$. By Theorem  \ref{thm:graph} and Theorem \ref{thm:R}, it suffices to prove that the simple random walk on $T$ is transient. The series law of electric networks implies that a discrete line segment of length $L$ with unit conductance has global conductance $1/L$, see \cite[Section 2.3]{LyonsPeres}. Hence, we may reduce the graph $T$ into an infinite $2$-ary tree $\hat T$ where all edges at depth $k$ have conductance $\varphi(k) := 1/L_k$, see \cite[Section 2.3]{LyonsPeres}. As already shown in the last paragraph of the proof of Theorem \ref{thm:pandazis-extended}, the associated Markov chain on $\hat T$ will be transient if (and only if) 
$$\sum_{k=1}^\infty (2^k \varphi(k))^{-1} < \infty.$$
Since $2^k \varphi(k) = k^\alpha $, the above sum is indeed finite if $\alpha > 1$.
\end{proof}


\begin{rem}\label{rk:line}
  We note that there is an improvement of statement (iii) where for any $p  < 2/3$, $|A \cap E_n| \geq p |E_n|$ for all $n$ large enough.  Indeed, instead of considering a perfect matching in the proof of Statement (iii), we could instead consider line ensembles. For integer $k \geq 1$, we say that $L = \cup_{i} L_i$ is a $k$-line ensemble if $L_i$ is a subset of $k$ edges forming a line segment and $L_i \cap L_j = \emptyset$ for all $i \ne j$. For $k=1$, we retrieve a matching. A $k$-line ensemble is perfect if all vertices are on some line segment. Arguing as in the proof of Theorem \ref{prop:shorten} , the Nash-Williams criterion implies that if $A_n  \subset L \cap E_n$ then $X \in O_G$. Moreover, we may check that any perfect $k$-line ensemble satisfies $|L \cap E_n| \sim (2 k / (3(k+1)))|E_n|$.
\end{rem}

\subsection{The type problem for the topological $\mathbb{Z}^d$-covers}

Consider a $\mathbb{Z}^d$-lattice in $\mathbb{R}^d$ and connect the points of $\mathbb{Z}^d$ that are at a distance $1$ by Euclidean arcs to make an embedded graph. Let $X$ be an infinite Riemann surface obtained by taking a uniform $\epsilon$-neighborhood of the above graph. For $d\geq 3$, the surface $X$ is not parabolic (see Lyons-Sullivan \cite{LyonsSullivan} or \cite{Saric-ft} for a different proof).

More generally, we consider in this subsection an infinite Riemann surface $X$ which has a pant decomposition which is roughly equivalent to a conductance graph on the $\mathbb Z^d$ lattice, for the definition of rough equivalence, see \cite[Section 2.6]{LyonsPeres}. The following theorem improves significantly on \cite[Corollary 9.6]{BHS}.

 \begin{thm}
 	\label{thm:BHS-extended}
Let $X$ be the infinite Riemann surface which has a pant decomposition $\mathcal G = (V_X,E_X,\ell_X)$ such that  $\sup_{e \in E_X} \ell_X(e)  < \infty$. For some $d \geq 3$, assume that $\mathcal G$ is roughly equivalent to the standard $\mathbb Z^d$-lattice in $\mathbb R^d$ whose edges are denoted by $E$ and equipped with conductances $(\ell(e))_{e \in E}$. For integer $n \geq 1$, let $E_n \subset E$ be the subset of edges whose mid-point has $L^2$-norm in $[n-1,n)$. If there exists a positive sequence $(\psi(n))_{n \in \mathbb N}$ with $\sum_n \psi(n)^{-1} < \infty$ such that  
  $$
  \hbox{ for all $n \geq 1$ and $e \in E_n$, }  \quad  \frac{\psi(n)}{n^{d-1}} \leq \ell(e) 
  $$  
    then $X \notin O_G$. Conversely, $X \in O_G$ if  $\sum_n \psi(n)^{-1} = \infty$ and 
       $$
  \hbox{ for all $n \geq 1$ and $e \in E_n$, }  \quad  \ell(e) \leq \frac{\psi(n)}{n^{d-1}}.
  $$    
  In particular, if some $C \geq 1$, for all $n \geq 1$, $e ,  f \in E_n$, $$ \frac 1 C   \leq   \frac{\ell(e)}{\ell(f)} \leq C$$ then $X \in O_G$ if and only if $\sum_n \ell(E_n)^{-1} = \infty$.
 \end{thm}
\begin{proof}
The proof is similar to the proof to Theorem \ref{thm:pandazis-extended}. By Theorem \ref{thm:graph} and Kanai's Theorem \cite[Theorem 2.17]{LyonsPeres}, it suffices to establish the type of the random walk on $\mathbb Z^d$ equipped with the conductances $(\ell(e))_{e \in E}$. Also by Rayleigh's principle, Theorem \ref{thm:R} it is sufficient to assume that $\ell(e) = \psi(n)/n^{d-1}$ for all $e \in E_n$. In what follows $C ,C'> 0$ are constants that can change from line to line.

If $\sum_n \psi(n)^{-1} = \infty$ then the claim is a direct consequence of Nash-Williams criterion, Theorem \ref{thm:NW}, applied to $\Pi_n = E_n$. Indeed, $E_n$ is a cutset since all edges of $e \in E$ embedded in $\mathbb R^d$ have Euclidean length $1$. Also, there are at most $C n^{d-1}$ elements in $E_n$. 

We now assume that $\sum_n \psi(n)^{-1} < \infty$. We give a proof using the random path method as explained in \cite{LyonsPeres}. Let $\theta \in S^{d-1}$ be a uniform Haar distributed random point of the sphere $S^{d-1}$ and consider the infinite random ray $R = \{ t \theta : t \geq 0 \} \subset \mathbb R^d$. Let $\Pi$ be a simple random path in $\mathbb Z^d$ which stays within distance $d$ of $R$. We may choose $\Pi$ measurably (for example, closest in Hausdorff distance). As explained in \cite[pp 40-41]{LyonsPeres}, a sufficient condition for transience is the condition: 
$$
\sum_{e \in E} \frac{\mathbb P (e \in \Pi)^2 }{\ell(e)}  < \infty,
$$
where $\mathbb P( \cdot) $ is the underlying probability measure. To this end, we observe that for $n \geq 1$ and $e = \{ x,y\} \in E_n$, if $u = x / \|x\|_2 \in S^{d-1}$, we have $\mathbb P ( e \in \Pi) \leq \mathbb P ( \| n \theta - n u \|_2 \leq C ) \leq C' / n^{d-1}$. Indeed, $\mathbb P ( \| \theta - u \|_2 \leq C/n )$ is independent of $u$ and we have $n^{d-1} /C'$ disjoints spherical caps of radius $C/n$ in $S^{d-1}$. Since $|E_n| \leq C n^{d-1}$, we get
$$
\sum_{e \in E} \frac{\mathbb P (e \in \Pi)^2 }{\ell(e)} = \sum_n \sum_{e \in E_n} \mathbb P (e \in \Pi)^2 \frac{ n^{d-1}}{\psi(n)} \leq  \sum_n C n^{d-1} \cdot \left(\frac{C'}{n^{d-1}} \right)^2 \cdot \frac{n^{d-1}}{\psi(n)}. 
$$
The conclusion follows.\end{proof}

We can also look for an analog of Theorem \ref{prop:shorten}  in our new setting. 

\begin{thm}
\label{prop:shorten2}
Let $d\geq 3$, $X, E, (E_n)_{n \in \mathbb N}$ and $(\ell(e))_{e \in E}$ be as in Theorem \ref{thm:BHS-extended}. Let $(\psi(n))_n$ be a positive sequence such that $\sum_n \psi(n)^{-1} = \infty$. For some $\alpha>0$ and all $n \in \mathbb N$, let $K(n) = \lceil n ( \log n )^\alpha \rceil$.  Finally, for $n \geq 1$, we denote by $A_n \subset E_n$ the set of $e \in  E_n$ such that 
$$
\ell_X( e)   > \frac{\psi(n)}{n^{d-1}}.
$$
The following holds: 
\begin{enumerate}[(i)]
\item If $\alpha \leq 1 $ and $ \psi(n) + |A_n| \leq K(n)$ for all $n$ large enough, then $X \in O_G$.

\item If $\alpha >  1$, there exists $A \subset E$ with $|A \cap E_n| \leq K(n)$ such that if $\ell(e) = 1$ for all $e \in A$ then $X \notin O_G$.

\item There exist $p > 0 $ and $A \subset E$  with $|A \cap E_n| \geq p |E_n |$ for all $n$ large enough such that, if $\limsup_n \psi(n+1)/\psi(n) < \infty$ and $A_n \subset A \cap E_n$, then $X \in O_G$.
\end{enumerate}
\end{thm}

\begin{proof}
The proof is very similar to the proof of Theorem \ref{prop:shorten} . Statement (i) is an immediate consequence of Nash-Williams criterion. For Statement (iii), it is enough to consider a matching $M =A$ such that $|A \cap E_n| \geq p |E_n |$ for all $n$ large enough. We should prove the existence of such matching for $p$ small enough. We can use the probabilistic method to do this, that is produce a random matching $M$ which satisfies with positive probability the event: for all $n$ large enough, $|M \cap E_n| \geq p |E_n|$.

Let $(U_e)_{e \in E}$ be independent and uniformly distributed random variable on $[0,1]$. We build a random matching $M$ by putting $ e \in M$ if $U_e  > \max_{ f \in N_e} U_f$ where $N_e$ is the set of edges which share an incident vertex with $e$. Since there are $|N_e| = 2 (2d-1)$ such edges, we have $\mathbb P ( e \in M) = q$ with $q = 1/(1 + (2(2d-1))$ and $\mathbb P(\cdot)$ is our underlying probability measure. If $\mathbb E(\cdot)$ denotes the expectation and $m_n = |M \cap E_n| = \sum_{e \in E_n} 1 ( e \in M)$ we deduce that $\mathbb E[ m_n ] = q |E_n|$. We can also compute the variance of $m_n$: 
\begin{align*}
\mathrm{Var}(m_n) & = \mathbb E (m_n^2)  - (\mathbb E m_n)^2  \\
& = \sum_{e,f \in E_n} \left( \mathbb P ( e \in M , f \in M)  - \mathbb P ( e \in M ) \mathbb P (  f \in M) \right).
    \end{align*}
Note that if $N_e \cap N_f = \emptyset$ then the events $\{e \in M\}$  and $\{f \in M \}$ are independent. Hence, 
\begin{align*}
\mathrm{Var}(m_n) & = \sum_{e \in E_n} \sum_{f \in E_n : N_e \cap N_f \ne \emptyset} \left( \mathbb P ( e \in M , f \in M)  - \mathbb P ( e \in M ) \mathbb P (  f \in M) \right)\\
& \leq C |E_n|, 
\end{align*}
where $C  = | \{ f : N_e \cap N_f \ne \emptyset \} | \leq 2 (2d)^2$. From Bienaymé–Chebyshev inequality, we find
$$
\mathbb P \left( m_n \leq \frac{\mathbb E[ m_n ]}{2} \right) \leq \mathbb P \left( | m_n - \mathbb E[ m_n ] | \geq  \frac{\mathbb E[ m_n ]}{2} \right) \leq \frac{ \mathrm{Var}(m_n) } { (\mathbb E[ m_n ]/2 )^2} \leq \frac{4 C}{q^2 |E_n|}.
$$
Since $|E_n| \geq  n^{d-1} / C$ for some $C >0$ and $d\geq 3$, we get that 
$$
\sum_n \mathbb P \left( m_n \leq \frac{\mathbb E[ m_n ]}{2} \right)  < \infty.
$$
From Borel-Cantelli lemma, it follows that with probability one, for all $n$ large enough, $m_n = |M_n \cap E_n| \geq (q/2) |E_n|$. It concludes the proof of Statement (iii) with $p = q/2$.


We finally prove Statement (ii). For some $n_1 \in \mathbb N$ and $0 < \epsilon < 1$ to be defined later, we set $$U = \left\{ (n + n_1, k_2,\ldots, k_d) : n \geq 0 , | k_i | \leq  \epsilon K(n)^{\frac{1}{d-1}}   \right\} \subset \mathbb Z^d.$$ We define $A \subset E$ as the set of edges such that both endpoints are in $U$. From \cite[Section 6]{Lyons83} and Rayleigh's principle, Theorem \ref{thm:R}, if $\alpha >1$ and $\ell(e) = 1$ for all $e \in A$ then $X \notin O_G$.

It thus remains to check that $|A \cap E_n | \leq K(n)$ if $n_1$ large enough. To this end, note that the norm of $(n + n_1, k_2,\ldots, k_d)\in U$ is in the interval $I = [n+n_1, n + n_1 +  d  \epsilon K(n)^{2/(d-1)} / ( n +n_1) ]$. If $d \geq 4$, we  have $I \subset [n+n_1 , n+n_1 + 1)$ if $n_1$ is large enough. We then deduce that for $\epsilon = 1/4$,  $|A \cap E_{n+n_1}| \leq 2 \lfloor 2 \epsilon K(n+1)^{1/(d-1)} \rfloor ^{d-1}  \leq 2 ( 2 \epsilon ) ^{d-1} K(n+1) \leq K(n+n_1)$. Also, for $m < n_1$, $A \cap E_m = \emptyset$. It concludes the proof for $d \geq 4$.

If $d = 3$, the above interval $I$ is contained in $[n+n_1,n+n_1 + m )$ with $m = \lceil d \epsilon (\log n)^\alpha \rceil $. Since $A$ is flat in $\mathbb R^3$, if $n_1$ is large enough for any $ n \geq n_1$ and $e  \in A \cap E_n$ we have that $e + (k,0,\ldots,0) \notin E_n$ for all $k \geq 2$.  We thus deduce that  $|A \cap E_{n+n_1}| \leq  2 \lfloor 2 \epsilon K(n+m)^{1/2} \rfloor ^{2} \leq 2 (2 \epsilon)^2 K(n+m) $. If $\epsilon =1/8$ and $n_1$ is large enough, $2 (2 \epsilon)^2 K(n+m)  \leq K (n+n_1)$ as requested.
\end{proof}

\begin{rem}\label{rk:pm}
We have used in the proof of Theorem \ref{prop:shorten2}(iii) the probabilistic method. It could have been possible to design an explicit deterministic matching. An advantage of the probabilistic method is that it extends to any other sequences $(E_n)$ of cut-sets such that $ |E_n| \sim C n^{d-1}$ (for example, for $p \in [1,\infty]$, $E_n$ could be the subset of edges whose mid-point has $L^p$-norm in $[n-1,n)$).  
\end{rem}

\vline

\author{Charles Bordenave, Institut de Mathématiques de Marseille, CNRS \& Aix-Marseille Université,  \em{charles.bordenave@cnrs.fr}}

\vspace{3pt}
\author{Xinlong Dong, Department of Mathematics,  Kingsborough Community College, CUNY, {\em Xinlong.Dong@kbcc.cuny.edu}}

\vspace{3pt}

\address{Dragomir \v Sari\' c, Current address: Institute for Advanced Study, 1 Einstein Drive, Princeton, NJ 08540, USA, {\em dsaric@ias.edu}}

\address{Dragomir \v Sari\' c, PhD Program in Mathematics, The Graduate Center, CUNY, 365 Fifth Ave., N.Y., N.Y., 10016 and 
Department of Mathematics, Queens College, CUNY, 65--30 Kissena Blvd., Flushing, NY 11367, USA, {\em Dragomir.Saric@qc.cuny.edu}}

\end{document}